\documentclass[11pt,reqno,twoside]{amsart}
\usepackage{amssymb,amsmath,amsthm,soul,color,paralist}
\usepackage{t1enc}
\usepackage[cp1250]{inputenc}
\usepackage{a4,indentfirst,latexsym}
\usepackage{graphics}
\usepackage{mathrsfs}
\usepackage{cite,enumitem,graphicx}
\usepackage[colorlinks=true,urlcolor=blue,
citecolor=red,linkcolor=blue,linktocpage,pdfpagelabels,
bookmarksnumbered,bookmarksopen]{hyperref}
\usepackage[english]{babel}
\usepackage[top=1.30cm, bottom=1.30cm, left=1.75cm, right=1.75cm]{geometry}
\usepackage[metapost]{mfpic}
\usepackage[hyperpageref]{backref}
\usepackage[colorinlistoftodos]{todonotes}
\usepackage[normalem]{ulem}

\numberwithin{equation}{section}

\newcommand{\R}{\mathbb{R}}

\newcommand{\N}{\mathcal{N}}
\newcommand{\h}{\mathcal{H}_{\e}}
\newcommand{\J}{\mathcal{J}}
\newcommand{\I}{\mathcal{I}}
\newcommand{\B}{\mathcal{B}}

\newcommand{\p}{2^{*}_{s}}
\newcommand{\ri}{\rightarrow}

\DeclareMathOperator{\supp}{supp}
\DeclareMathOperator{\e}{\varepsilon}

\newtheorem{lem}{Lemma}[section]
\newtheorem{thm}{Theorem}[section]

\title[fractional critical Kirchhoff equations]{Concentrating solutions for a fractional Kirchhoff equation with critical growth}

\author[V. Ambrosio]{Vincenzo Ambrosio}
\address{Vincenzo Ambrosio\hfill\break\indent 
Dipartimento di Ingegneria Industriale e Scienze Matematiche \hfill\break\indent
Universit\`a Politecnica delle Marche\hfill\break\indent
Via Brecce Bianche, 12\hfill\break\indent
60131 Ancona (Italy)}
\email{v.ambrosio@univpm.it}


\date{}

\begin{document}

\begin{abstract}
In this paper we consider the following class of fractional Kirchhoff equations with critical growth:
\begin{equation*}
\left\{
\begin{array}{ll}
\left(\e^{2s}a+\e^{4s-3}b\int_{\R^{3}}|(-\Delta)^{\frac{s}{2}}u|^{2}dx\right)(-\Delta)^{s}u+V(x)u=f(u)+|u|^{2^{*}_{s}-2}u \quad &\mbox{ in } \R^{3}, \\
u\in H^{s}(\R^{3}), \quad u>0 &\mbox{ in } \R^{3}, 
\end{array}
\right. 
\end{equation*}
where $\e>0$ is a small parameter, $a, b>0$ are constants, $s\in (\frac{3}{4}, 1)$, $2^{*}_{s}=\frac{6}{3-2s}$ is the fractional critical exponent, $(-\Delta)^{s}$ is the fractional Laplacian operator, $V$ is a positive continuous potential and $f$ is a superlinear continuous function with subcritical growth.
Using penalization techniques and variational methods, we prove the existence of a family of positive solutions $u_{\e}$ which concentrates around a local minimum of $V$ as $\e\ri 0$.
\end{abstract}
\maketitle

\section{Introduction}

This paper is devoted to the existence and concentration of positive solutions for the following fractional Kirchhoff type equation with critical nonlinearity:
\begin{equation}\label{P}
\left\{
\begin{array}{ll}
\left(\e^{2s}a+\e^{4s-3}b\int_{\R^{3}}|(-\Delta)^{\frac{s}{2}}u|^{2}dx\right)(-\Delta)^{s}u+V(x)u=f(u)+|u|^{2^{*}_{s}-2}u &\mbox{ in } \R^{3}, \\
u\in H^{s}(\R^{3}), \quad u>0 &\mbox{ in } \R^{3}, 
\end{array}
\right. 
\end{equation}
where $\e>0$ is a small parameter, $a, b>0$ are constants, $s\in (\frac{3}{4}, 1)$ is fixed,  $2^{*}_{s}=\frac{6}{3-2s}$ is the fractional critical exponent, and $(-\Delta)^{s}$ is the fractional Laplacian operator, which (up to normalization factors) may be defined for smooth functions $u:\R^{3}\rightarrow \R$ as
$$
(-\Delta)^{s}u(x)=-\frac{1}{2} \int_{\R^{3}} \frac{u(x+y)+u(x-y)-2u(x)}{|y|^{3+2s}} dy \quad (x\in \R^{3}),
$$
(see \cite{DPV, MBRS} and the references therein for further details and applications).\\
The potential $V:\R^{3}\rightarrow \R$ is a continuous function satisfying the following conditions introduced by del Pino and Felmer in \cite{DF}:
\begin{compactenum}[$(V_1)$]
\item $V_{1}:=\inf_{x\in \R^{3}} V(x)>0$,
\item there exists a bounded open set $\Lambda\subset \R^{3}$ such that
$$
0<V_{0}:=\inf_{\Lambda} V<\min_{\partial \Lambda} V,
$$
\end{compactenum}
while $f:\R\rightarrow \R$ is a continuous function fulfilling the following hypotheses:
\begin{compactenum}[$(f_1)$]
\item $f(t)=o(t^{3})$ as $t\rightarrow 0^{+}$,
\item there exist $q, \sigma\in (4, 2^{*}_{s})$, $C_{0}>0$ such that
$$
f(t)\geq C_{0} t^{q-1} \quad \forall t>0, \quad \lim_{t\rightarrow \infty} \frac{f(t)}{t^{\sigma-1}}=0,
$$
\item there exists $\vartheta\in (4, 2^{*}_{s})$ such that $0<\vartheta F(t)\leq t f(t)$ for all $t>0$,
\item the map $t\mapsto \frac{f(t)}{t^{3}}$ is increasing in $(0, \infty)$.
\end{compactenum}
Since we will look for positive solutions to \eqref{P}, we assume that $f(t)=0$ for $t\leq 0$.

We note that when $a=1$, $b=0$ and $\R^{3}$ is replaced by $\R^{N}$, then \eqref{P} reduces to a fractional Schr\"odinger equation of the type
\begin{equation}\label{FSE}
\e^{2s}(-\Delta)^{s}u+V(x)u=h(x, u) \mbox{ in } \R^{N},
\end{equation}
which has been introduced by Laskin \cite{Laskin1}  as a result of expanding the Feynman path integral, from the Brownian like to the L\'evy like quantum mechanical paths. Equation \eqref{FSE} has received a great interest by many mathematicians, and several results have been obtained under different and suitable assumptions on $V$ and $h$; see for instance \cite{Adie, AM, A1, A0, DDPW, DMV, FFV, FQT, I1, Secchi1, SZY} and the references therein. In particular way, the existence and concentration as $\e\ri 0$ of positive solutions to \eqref{FSE} has been widely investigated in recent years. For instance, D\'avila et al. \cite{DDPW} showed via Lyapunov-Schmidt reduction, that if the potential $V$ satisfies
$$
V\in C^{1, \alpha}(\R^{N})\cap L^{\infty}(\R^{N}) \mbox{ and } \inf_{x\in \R^{N}} V(x)>0,
$$
then (\ref{P}) has multi-peak  solutions.
Shang et al. \cite{SZY} used  Ljusternik-Schnirelmann theory to obtain multiple positive solutions for a fractional Schr\"odinger equation with critical growth
assuming that the potential $V: \R^{N}\rightarrow \R$ fulfills the following assumption proposed by Rabinowitz \cite{Rab}:
\begin{equation*}
V_{\infty}:=\liminf_{|x|\rightarrow \infty} V(x)>\inf_{x\in \R^{N}} V(x)=:V_{1}, \mbox{ where } V_{\infty}\in (0, \infty] \tag{V}.
\end{equation*}
Fall et al. \cite{FFV} established necessary and sufficient conditions on the smooth potential $V$ in order to produce concentration of solutions of (\ref{P}) when the parameter $\e$ converges to zero.  Moreover, when $V$ is coercive and has a unique global minimum, then ground-states concentrate at this point.
Alves and Miyagaki \cite{AM} (see also \cite{A1}) studied the existence and concentration of positive solutions to (\ref{P}), via a penalization approach, under assumptions $(V_1)$-$(V_2)$ and $f$ is a subcritical nonlinearity.

On the other hand, if we set $s=\e=1$ and we replace $f(u)+|u|^{\p-2}u$ by a more general nonlinearity $h(x, u)$, then \eqref{P} becomes the well-known classical Kirchhoff equation 
\begin{equation}\label{SKE}
-\left(a+b\int_{\R^{3}} |\nabla u|^{2}dx \right)\Delta u+V(x)u=h(x,u) \quad \mbox{ in } \R^{3},
\end{equation}
which is related to the stationary analogue of the Kirchhoff equation
\begin{equation}\label{KE}
\rho u_{tt} - \left( \frac{p_{0}}{h}+ \frac{E}{2L}\int_{0}^{L} |u_{x}|^{2} dx \right) u_{xx} =0,
\end{equation}
introduced by Kirchhoff \cite{Kir} in $1883$ as an extension of the classical D'Alembert's wave equation for describing  the transversal oscillations of a stretched string. Here  $L$ is the length of the string, $h$ is the area of the cross-section, $E$ is the young modulus (elastic modulus) of the material, $\rho$ is the mass density, and $p_{0}$ is the initial tension.
We refer to \cite{B, P} for the early classical studies dedicated to \eqref{KE}. We also note that nonlocal boundary value problems like \eqref{SKE} model several physical and biological systems where $u$ describes a process which depends on the average of itself, as for example, the population density; see \cite{ACM, ChL}.
However, only after the Lions' work \cite{LionsK}, where a functional analysis approach was proposed to attack a general Kirchhoff equation in arbitrary dimension with external force term, problem \eqref{SKE} began to catch the attention of several mathematicians; see \cite{ACF, CKW, FJ, HLP, HZ, WTXZ} and the references therein.
For instance, He and Zou \cite{HZ} obtained existence and multiplicity results for small $\e>0$ of the following perturbed Kirchhoff equation
\begin{align}\label{CKE}
-\left(a\varepsilon^{2}+b\varepsilon \int_{\R^{3}}|\nabla u|^{2} dx\right)\Delta u+V(x)u=g(u) \quad \mbox{ in } \mathbb{R}^{3},
\end{align} 
where the potential $V$ satisfies condition $(V)$ and $g$ is a subcritical nonlinearity. Wang et al. \cite{WTXZ} studied the multiplicity and concentration phenomenon for \eqref{CKE} when $g(u)=\lambda f(u)+|u|^{4}u$, $f$ is a continuous subcritical nonlinearity and $\lambda$ is large. Figueiredo and Santos Junior \cite{FJ} used the generalized Nehari manifold method to obtain a multiplicity result for a subcritical Kirchhoff equation under conditions $(V_1)$-$(V_2)$.
He et al. \cite{HLP} dealt with the existence and multiplicity of solutions to \eqref{CKE}, where $g(u)=f(u)+u^{5}$, $f\in C^{1}$ is a subcritical nonlinearity which does not satisfies the Ambrosetti-Rabinowitz condition \cite{AR} and $V$ fulfills $(V_1)$-$(V_2)$.

In the nonlocal framework, Fiscella and Valdinoci \cite{FV}  proposed for the first time a stationary fractional Kirchhoff variational model in a bounded domain $\Omega\subset \R^{N}$ with homogeneous Dirichlet boundary conditions and involving a critical nonlinearity:
\begin{align}\label{FKE}
\left\{
\begin{array}{ll}
M\left(\int_{\R^{N}}|(-\Delta)^{\frac{s}{2}}u|^{2}dx\right)(-\Delta)^{s}u=\lambda f(x, u)+|u|^{\p-2}u \quad &\mbox{ in } \Omega,\\
u=0 &\mbox{ in } \R^{N}\setminus \Omega, 
\end{array}
\right. 
\end{align}
where $M$ is a continuous Kirchhoff function whose model case is given by $M(t)=a+bt$.  
Their model takes care of the nonlocal aspect of the tension arising from nonlocal measurements of the fractional length of the string; see \cite{FV} for more details.
After the pioneering work \cite{FV}, several authors dealt with existence and multiplicity of solutions for \eqref{FKE}; see \cite{AFP, FMBS, MBRS, Ny} and their references. On the other hand, some interesting results for fractional Kirchhoff equations in $\R^{N}$ have been established in \cite{AI1, AI2, FP, LSZ, MRZ, PuSa, PXZ}.
For instance, Pucci and Saldi \cite{PuSa} obtained the existence and multiplicity of nontrivial solutions for a Kirchhoff type eigenvalue problem in $\R^N$ involving a critical nonlinearity.
Fiscella and Pucci \cite{FP}  dealt with stationary fractional Kirchhoff $p$-Laplacian equations involving critical Hardy-Sobolev nonlinearities and nonnegative potentials.
In \cite{AI1} a multiplicity result for a fractional Kirchhoff equation involving a Beresticky-Lions type nonlinearity is proved.
The author and Isernia \cite{AI2} used penalization method and  Lusternik-Schnirelmann category theory to study the existence and multiplicity of solutions for a fractional Schr\"odinger-Kirchhoff equation with subcritical nonlinearities; 
see also \cite{HZm} in which the authors used the approach in \cite{AI2} to consider a subcritical version of \eqref{P}.
Liu et al. \cite{LSZ}, via the monotonicity trick and the profile decomposition, proved the existence of ground states to a fractional Kirchhoff equation with critical nonlinearity in low dimension.

Motivated by the above works, in this paper we aim to study the existence and concentration behavior of solutions to \eqref{P} under assumptions $(V_1)$-$(V_2)$ and $(f_1)$-$(f_4)$. 
More precisely, our main result can be stated as follows:
\begin{thm}\label{thm1}
Assume that $(V_1)$-$(V_2)$ and $(f_1)$-$(f_4)$ hold. Then, 
there exists $\e_{0}>0$ such that, for each  $\e\in (0, \e_{0})$, problem \eqref{P} has a positive solution $u_{\e}$. Moreover, if $\eta_{\e}$ denotes a global maximum point of $u_{\e}$, then we have
$$
\lim_{\e\rightarrow 0} V(\eta_{\e})=V_{0},
$$
and there exists a constant $C >0$ such that
$$
0<u_{\varepsilon}(x)\leq \frac{C \varepsilon^{3+2s}}{\varepsilon^{3+2s}+ |x-\eta_{\varepsilon}|^{3+2s}} 
\quad \mbox{ for all } x\in \R^{3}. 
$$
\end{thm}

The proof of Theorem \ref{thm1} will be done via appropriate variational arguments. After considering the $\e$-rescaled problem associated with \eqref{P}, 
we use a variant of the penalization technique introduced in \cite{DF} (see also \cite{ADOS, FF}) which  consists in modifying in a suitable way the nonlinearity outside $\Lambda$, solving a modified problem and then check that, for $\e>0$ small enough, the solutions of the modified problem are indeed solutions of the original one. These solutions will be obtained as critical points of the modified energy functional $\J_{\e}$ which, in view of the growth assumptions on $f$ and the auxiliary nonlinearity, possesses a mountain pass geometry \cite{AR}. In order to recover some compactness properties for $\J_{\e}$, we have to circumvent several difficulties which make our study rather delicate. The first one is related to the presence of the Kirchhoff term in \eqref{P} which does not permit to verify in a standard way that if $u$ is the weak limit of a Palais-Smale sequence ($(PS)$ in short) $\{u_{n}\}_{n\in \mathbb{N}}$ for $\J_{\e}$, then $u$ is a weak solution for the modified problem. 
The second one is due to the lack of compactness caused by  the unboundedness of the domain $\R^{3}$ and the critical Sobolev exponent.
Anyway, we will be able to overcome these problems looking for critical points of a suitable functional whose quadratic part involves the limit term of $(a+b[u_{n}]^{2}_{s})$, and showing that the mountain pass level $c_{\e}$ of $\J_{\e}$ is strictly less than a  threshold value related to the best constant of the embedding $H^{s}(\R^{3})$ in $L^{\p}(\R^{3})$.
Then, applying mountain pass lemma, we will deduce the existence of a positive solution for the modified problem.
Finally, combining a compactness argument with a Moser iteration procedure \cite{Moser}, we prove that the solution of the modified problem is also a solution to the original one for $\e>0$ small enough, and that it decays at zero at infinity with polynomial rate.
To our knowledge, this is the first time that concentration phenomenon for problem \eqref{P} is investigated in the literature.

\noindent
The paper is organized as follows: in Section $2$ we introduce the modified problem and we provide some technical results. In Section $3$ we give the proof of Theorem \ref{thm1}.

\section{The modified problem}

\subsection{Preliminaries}\hfill\\
\noindent
Here we fix the notations and we recall some useful preliminary results on fractional Sobolev spaces  (see also \cite{DPV, MBRS} for more details).\\
If $A\subset \R^{3}$, we denote by $|u|_{L^{q}(A)}$ the $L^{q}(A)$-norm of a function $u:\R^{3}\rightarrow \R$, and by $|u|_{q}$ its $L^{q}(\R^{3})$-norm. We denote by $\B_{r}(x)$ the ball centered at $x\in \R^{3}$ with radius $r>0$. When $x=0$, we put $\B_{r}=\B_{r}(0)$.
Let us define $\mathcal{D}^{s,2}(\R^{3})$ as the completion of $C^{\infty}_{c}(\R^{3})$ with respect to the norm
$$
[u]^{2}_{s}:=\iint_{\R^{6}} \frac{|u(x)-u(y)|^{2}}{|x-y|^{3+2s}} dx dy =\int_{\R^{3}}|(-\Delta)^{\frac{s}{2}}u|^{2}dx,
$$
where the second identity holds up to a constant; see \cite{DPV}.
Then we consider the fractional Sobolev space
$$
H^{s}(\R^{3}):=\Bigl\{u\in L^{2}(\R^{3}): [u]_{s}<\infty\Bigr\}
$$
endowed with the norm
$$
\|u\|^{2}:=[u]^{2}_{s}+|u|_{2}^{2}.
$$
We recall the following main embeddings for the fractional Sobolev spaces:
\begin{thm}\cite{DPV}\label{Sembedding}
Let $s\in (0,1)$. Then there exists a sharp constant $S_{*}=S_{*}(s)>0$ such that for any $u\in \mathcal{D}^{s,2}(\R^{3})$
\begin{equation*}
|u|^{2}_{\p} \leq S_{*}^{-1}[u]^{2}_{s}.
\end{equation*}
Moreover, $H^{s}(\R^{3})$ is continuously embedded in $L^{p}(\R^{3})$ for any $p\in [2, 2^{*}_{s}]$ and compactly in $L^{p}_{loc}(\R^{3})$ for any $p\in [1, 2^{*}_{s})$. 
\end{thm}

\noindent
The following lemma is a version of the well-known Lions type result:
\begin{lem}\cite{FQT}\label{Lions}
If $\{u_{n}\}_{n\in \mathbb{N}}$ is a bounded sequence in $H^{s}(\R^{3})$ and if
$$
\lim_{n \rightarrow \infty} \sup_{y\in \R^{3}} \int_{B_{R}(y)} |u_{n}|^{2} dx=0
$$
for some $R>0$, then $u_{n}\rightarrow 0$ in $L^{r}(\R^{3})$ for all $r\in (2, 2^{*}_{s})$.
\end{lem}

\noindent
We also recall the following useful technical result.
\begin{lem}\cite{PP}\label{pp}
Let $u\in \mathcal{D}^{s, 2}(\R^{3})$. Let $\varphi\in C^{\infty}_{c}(\R^{3})$ and for each $r>0$ we define $\varphi_{r}(x)=\varphi(x/r)$. Then, $[u \varphi_{r}]_{s} \rightarrow 0$ as $r\rightarrow 0$.
If in addition $\varphi=1$ in a neighborhood of the origin, then
$[u \varphi_{r}]_{s} \rightarrow [u]_{s}$ as $r\rightarrow \infty$.
\end{lem}

\subsection{Functional Setting}\hfill\\
In order to study \eqref{P}, we use the change of variable $x\mapsto \e x$ and we will look for solutions to
\begin{equation}\label{Pe}
\left\{
\begin{array}{ll}
(a+b[u]^{2}_{s})(-\Delta)^{s}u+V(\e x)u=f(u)+|u|^{2^{*}_{s}-2}u \quad &\mbox{ in } \R^{3}, \\
u\in H^{s}(\R^{3}), \quad u>0 &\mbox{ in } \R^{3}.
\end{array}
\right. 
\end{equation}

Now, we introduce a penalization method in the spirit of \cite{DF} which will be fundamental to obtain our main result.
First of all,  without loss of generality, we will assume that
$$
0\in \Lambda \mbox{ and } V(0)=V_{0}=\inf_{\Lambda} V.
$$
Let $K>\frac{2\vartheta}{\vartheta-2}$ and $a_{0}>0$ be such that 
\begin{align}\label{V0Ka}
f(a_{0})+a_{0}^{2^{*}_{s}-1}=\frac{V_{1}}{K}a_{0}
\end{align} 
and we define
\begin{equation*}
\tilde{f}(t):=\left\{
\begin{array}{ll}
f(t)+(t^+)^{2^{*}_{s}-1} &\mbox{ if } t\leq a_{0},\\
\frac{V_{1}}{K}t &\mbox{ if } t>a_{0},
\end{array}
\right.
\end{equation*}
and 
\begin{equation*}
g(x, t):=\left\{
\begin{array}{ll}
\chi_{\Lambda}(x)(f(t)+(t^{+})^{2^{*}_{s}-1})+(1-\chi_{\Lambda}(x))\tilde{f}(t) &\mbox{ if } t> 0,\\
0 &\mbox{ if } t\leq 0.
\end{array}
\right.
\end{equation*}

It is easy to check that $g$ satisfies the following properties:
\begin{compactenum}[$(g_1)$]
\item $\lim_{t\rightarrow 0^{+}} \frac{g(x,t)}{t^{3}}=0$ \quad uniformly with respect to $x\in \R^{3}$,
\item $g(x,t)\leq f(t)+t^{2^{*}_{s}-1}$ \quad for all $x\in \R^{3}$, $t>0$,
\item $(i)$ $0\leq \vartheta G(x,t)<g(x,t)t$ \quad for all $x\in \Lambda$ and $t>0$,\\
$(ii)$ $0\leq 2G(x,t)<g(x,t)t\leq \frac{V_{1}}{K}t^{2}$ \quad for all $x\in \R^{3}\setminus\Lambda$ and $t>0$,
\item for each $x\in \Lambda$ the function $\frac{g(x,t)}{t^{3}}$ is increasing in $(0, \infty)$, and for each $x\in \R^{3}\setminus\Lambda$ the function $\frac{g(x,t)}{t^{3}}$ is increasing in $(0, a_{0})$. 
\end{compactenum}
Then, we consider the following modified problem
\begin{equation}\label{MP}
\left\{
\begin{array}{ll}
(a+b[u]^{2}_{s})(-\Delta)^{s}u+V(\e x)u=g(\e x, u) \quad &\mbox{ in } \R^{3}, \\
u\in H^{s}(\R^{3}), \quad u>0 &\mbox{ in } \R^{3}. 
\end{array}
\right. 
\end{equation}
The corresponding energy functional is given by 
$$
\J_{\e}(u)=\frac{1}{2}\|u\|^{2}_{\e}+\frac{b}{4}[u]^{4}_{s}-\int_{\R^{3}} G(\e x, u)\, dx,
$$
which is well-defined on the space
$$
\h:=\left\{u\in H^{s}(\R^{3}): \int_{\R^{3}} V(\e x)u^{2} \,dx<\infty\right\}
$$
endowed with the norm
$$
\|u\|_{\e}^{2}:=a[u]^{2}_{s}+\int_{\R^{3}} V(\e x)u^{2} \, dx.
$$
Clearly $\h$ is a Hilbert space with the following inner product 
$$
(u, v)_{\e}:=a\iint_{\R^{6}} \frac{(u(x)-u(y))(v(x)-v(y))}{|x-y|^{3+2s}} \, dxdy+\int_{\R^{3}} V(\e x)uv \, dx.
$$
It is standard to show that $\J_{\e}\in C^{1}(\h, \R)$ and its differential is given by
$$
\langle \J'_{\e}(u),v\rangle=(u, v)_{\e}+b[u]^{2}_{s}\iint_{\R^{6}} \frac{(u(x)-u(y))(v(x)-v(y))}{|x-y|^{3+2s}} \, dxdy    -\int_{\R^{3}} g(\e x, u)v\, dx
$$
for any $u,v\in \h$. 
Let us introduce the Nehari manifold associated with \eqref{MP}, that is,
$$
\mathcal{N}_{\e}:=\Bigl\{u\in \h\setminus\{0\}: \langle \J'_{\e}(u),u\rangle=0\Bigr\}.
$$

\noindent
We begin by proving that $\J_{\e}$ possesses a nice geometric structure:
\begin{lem}\label{lem2.2}
The functional $\J_{\e}$ has a mountain-pass geometry:
\begin{compactenum}[$(a)$]
\item there exist $\alpha, \rho>0$ such that $\mathcal{J}_{\e}(u) \geq \alpha $ with $\|u\|_{\e}= \rho$; 
\item there exists $e\in \mathcal{H}_{\e}$ with $\|e\|_{\e}>\rho$  such that $\mathcal{J}_{\e}(e)<0$.
\end{compactenum}
\end{lem}

\begin{proof}
$(a)$ By assumptions $(g_1)$ and $(g_2)$ we deduce that for any $\xi>0$ there exists $C_{\xi}>0$ such that
\begin{equation*}
\mathcal{J}_{\e}(u)\geq\frac{1}{2}\|u\|_{\e}^{2}- \int_{\R^{3}} G(\e x, u)\, dx \geq \frac{1}{2} \|u\|_{\e}^{2} -  \xi C\|u\|_{\e}^{2}- C_{\xi}C\|u\|_{\e}^{2^{*}_{s}}. 
\end{equation*}
Then, there exist $\alpha, \rho>0$ such that $\mathcal{J}_{\e}(u) \geq \alpha $ with $\|u\|_{\e}= \rho$.\\
$(b)$ Using $(g_3)$-$(i)$, we deduce that for any $u\in C^{\infty}_{c}(\R^{3})\setminus\{0\}$ such that $u\geq 0$ and $supp(u)\subset \Lambda_{\e}$, and for all $\tau>0$ it holds
\begin{align}\label{b2.2}
\mathcal{J}_{\e}(\tau u)&= \frac{\tau^{2}}{2} \|u\|_{\e}^{2}+ b\frac{\tau^{4}}{4} [u]^{4}_{s} -\int_{\Lambda_{\e}} G(\e x, \tau u)\, dx\nonumber\\
&\leq \frac{\tau^{2}}{2} \|u\|_{\e}^{2}+ b\frac{\tau^{4}}{4} [u]^{4}_{s} - C_{1} \tau^{\vartheta}\int_{\Lambda_{\e}} u^{\vartheta} \, dx + C_{2}, 
\end{align}
for some constants $C_1, C_{2}>0$. Recalling that $\vartheta \in (4,2^{*}_{s})$ we can conclude that $\mathcal{J}_{\e}(\tau u)\rightarrow -\infty \mbox{ as } \tau\rightarrow \infty$. 
\end{proof}

\noindent
In view of Lemma \ref{lem2.2}, we can use a variant of the mountain-pass theorem without $(PS)$-condition (see \cite{Willem}) 
to deduce the existence of a Palais-Smale sequence $\{u_{n}\}_{n\in \mathbb{N}}\subset \h$ such that
\begin{align}\label{PSu}
\J_{\e}(u_{n})=c_{\e}+o_{n}(1) \quad \mbox{ and } \quad \J'_{\e}(u_{n})=o_{n}(1)
\end{align}
where
\begin{align}\label{minimax}
c_{\e}:=\inf_{\gamma\in \Gamma_{\e}} \max_{t\in [0,1]} \J_{\e}(\gamma(t)) \quad \mbox{ and } \quad \Gamma_{\e}:=\Bigl\{\gamma\in C([0, 1], \h): \gamma(0)=0, \J_{\e}(\gamma(1))\leq 0\Bigr\}.
\end{align}
As in \cite{Willem}, we can use the following equivalent characterization of $c_{\e}$ more appropriate for our aim:
$$
c_{\e}=\inf_{u\in \h\setminus\{0\}} \max_{t\geq 0} \J_{\e}(t u).
$$
Moreover, from the monotonicity of $g$, it is easy to see that for all $u\in \h\setminus\{0\}$ there exists a unique $t_{0}=t_{0}(u)>0$ such that 
$$
\J_{\e}(t_{0}u)=\max_{t\geq 0} \J_{\e}(tu).
$$
In the next lemma, we will see that $c_{\e}$ is less then a threshold value involving the best constant $S_{*}$ of Sobolev embedding $\mathcal{D}^{s,2}(\R^{3})$ in $L^{\p}(\R^{3})$. More precisely:
\begin{lem}\label{lem444}
There exists $T>0$ such that
$$
c_{\e}<\frac{a}{2}S_{*}T^{3-2s}+\frac{b}{4}S_{*}^{2}T^{6-4s}-\frac{1}{2^{*}_{s}}T^{3}=:c_{*}
$$
for all $\e>0$.
\end{lem}
\begin{proof}
We argue as in \cite{LSZ}.
Let $\eta\in C^{\infty}_{c}(\R^{3})$ be a cut-off function such that $\eta=1$ in $\B_{\rho}$, $\supp(\eta)\subset \B_{2\rho}$ and $0\leq \eta\leq 1$, where $\B_{2\rho}\subset \Lambda_{\e}$. For simplicity, we assume that $\rho=1$.  
We know (see \cite{CT}) that $S_{*}$ is achieved by $U(x)=\kappa(\mu^{2}+|x-x_{0}|^{2})^{-\frac{3-2s}{2}}$, with $\kappa\in \R$, $\mu>0$ and $x_{0}\in \R^{3}$. Taking $x_{0}=0$, as in \cite{SV2}, we can define 
$$
v_{h}(x):=\eta(x)u_{h}(x) \quad \forall h>0,
$$ 
where 
$$
u_{h}(x):=h^{-\frac{3-2s}{2}}u^{*}(x/h) \mbox{ and } u^{*}(x):=\frac{U(x/S_{*}^{\frac{1}{2s}})}{|U|_{2^{*}_{s}}}.
$$
Then $(-\Delta)^{s}u_{h}=|u_{h}|^{2^{*}_{s}-2}u_{h}$ in $\R^{3}$ and $[u_{h}]^{2}_{s}=|u_{h}|^{2^{*}_{s}}_{2^{*}_{s}}=S_{*}^{\frac{3}{2s}}$. We also recall the following useful estimates: 
\begin{align}
&A_{h}:=[v_{h}]^{2}_{s}=S^{\frac{3}{2s}}_{*}+O(h^{3-2s}) \label{sq1}\\
&B_{h}:=|v_{h}|_{2}^{2}=O(h^{3-2s}) \label{sq2}\\
&C_{h}:=|v_{h}|_{q}^{q} \geq 
\left\{
\begin{array}{ll}
O(h^{3-\frac{(3-2s)q}{2}}) &\mbox{ if } q>\frac{3}{3-2s}\\
O(\log (\frac{1}{h})h^{3-\frac{(3-2s)q}{2}}) &\mbox{ if } q=\frac{3}{3-2s} \\
O(h^{\frac{(3-2s)q}{2}}) &\mbox{ if } q<\frac{3}{3-2s}
\end{array}
\right.\label{sq3}\\
&D_{h}:=|v_{h}|_{2^{*}_{s}}^{2^{*}_{s}}=S^{\frac{3}{2s}}_{*}+O(h^{3}). \label{sq4}
\end{align}
Let us note that for all $h>0$ there exists $t_{0}>0$ such that $\J_{\e}(\gamma_{h}(t_{0}))<0$, where $\gamma_{h}(t)=v_{h}(\cdot/t)$.
Indeed, setting $V_{2}:=\max_{x\in \overline{\Lambda}} V(x)$, by $(f_2)$ we have
\begin{align}\label{sq5}
\J_{\e}(\gamma_{h}(t))&\leq \frac{a}{2}t^{3-2s}[v_{h}]^{2}_{s}+\frac{V_{2}}{2} t^{3} |v_{h}|_{2}^{2}+\frac{b}{4}t^{6-4s}[v_{h}]^{4}_{s}-\frac{t^{3}}{2^{*}_{s}}|v_{h}|^{2^{*}_{s}}_{2^{*}_{s}}-\frac{t^{3}}{q}|v_{h}|_{q}^{q}C_{0} \nonumber \\
&=\frac{a}{2}t^{3-2s}A_{h}+\frac{b}{4}A_{h}^{2}t^{6-4s}+\left(V_{2}\frac{B_{h}}{2}-\frac{D_{h}}{2^{*}_{s}}-\frac{C_{0}C_{h}}{q}\right) t^{3}.
\end{align}
Since $0<6-4s<3$, we can use \eqref{sq2} to deduce that
$$
V_{2}\frac{B_{h}}{2}-\frac{D_{h}}{2^{*}_{s}}\rightarrow -\frac{1}{2^{*}_{s}} S_{*}^{\frac{3}{2s}}
$$
as $h\rightarrow 0$. Hence, using \eqref{sq1}, we can see that for all $h>0$ sufficiently small $\J_{\e}(\gamma_{h}(t))\rightarrow -\infty$ as $t\rightarrow \infty$, that is there exists $t_{0}>0$ such that $\J_{\e}(\gamma_{h}(t_{0}))<0$. 

Now, as $t\rightarrow 0^{+}$, we have 
\begin{align*}
[\gamma_{h}(t)]^{2}_{s} + |\gamma_{h}(t)|_{2}^{2}  = t^{3-2s} A_{h} + t^{3}B_{h}\rightarrow 0 \mbox{ uniformly for } h>0 \mbox{ small. }
\end{align*}
We set $\gamma_{h}(0)=0$. Then $\gamma_{h}(t_{0}\cdot) \in \Gamma_{\e}$, where $\Gamma_{\e}$ is defined as in \eqref{minimax} and we infer that 
$$
c_{\e}\leq \sup_{t\geq 0} \J_{\e}(\gamma_{h}(t)).
$$

Taking into account that $c_{\e}>0$, by \eqref{sq5} there exists $t_{h}>0$ such that 
\begin{align*}
\sup_{t\geq 0} \J_{\e}(\gamma_{h}(t))= \J_{\e}(\gamma_{h}(t_{h})). 
\end{align*}
In the light of \eqref{sq1}, \eqref{sq3} and \eqref{sq5} we deduce that $\J_{\e}(\gamma_{h}(t)) \rightarrow 0^{+}$ as $t\rightarrow 0^{+}$ and $\J_{\e}(\gamma_{h}(t))\rightarrow -\infty$ as $t\rightarrow \infty$ uniformly for $h>0$ small. 
Then there exist $t_{1}, t_{2}>0$ (independent of $h>0$) satisfying $t_{1}\leq t_{h}\leq t_{2}$. 

Set 
\begin{align*}
H_{h}(t):= \frac{aA_{h}}{2} t^{3-2s} + \frac{bA_{h}^{2}}{4}t^{6-4s} -\frac{D_{h}}{2^{*}_{s}} t^{3}. 
\end{align*}
Therefore, 
\begin{align*}
c_{\e}\leq \sup_{t\geq 0} H_{h}(t) + \left(\frac{V_{2}B_{h}}{2} -\frac{C_{0}C_{h}}{q}\right) t_{h}^{3}. 
\end{align*}
From \eqref{sq3}, for any $q\in (2, 2^{*}_{s})$, we have $C_{h}\geq O(h^{3-\frac{(3-2s)q}{2}})$. Then, by \eqref{sq2}, we can infer
\begin{align*}
c_{\e}\leq \sup_{t\geq0} H_{h}(t) + O(h^{3-2s}) - O(C_{0} h^{3-\frac{(3-2s)q}{2}}). 
\end{align*}
Since $3-2s>0$ and $3-\frac{(3-2s)q}{2}>0$, we obtain 
\begin{align*}
\sup_{t\geq0} H_{h}(t) \geq \frac{c_{\e}}{2} \quad \mbox{ uniformly for } h>0 \mbox{ small. }
\end{align*}
Arguing as above, there exist $t_{3}, t_{4}>0$ (independent of $h>0$) such that 
\begin{align*}
\sup_{t\geq0} H_{h}(t)= \sup_{t\in [t_{3}, t_{4}]} H_{h}(t). 
\end{align*}
By \eqref{sq1} we deduce 
\begin{align}\label{sq6}
c_{\e}\leq \sup_{t\geq0} K(S_{*}^{\frac{1}{2s}} t) +O(h^{3-2s}) - O(C_{0} h^{3-\frac{(3-2s)q}{2}}), 
\end{align}
where 
\begin{align*}
K(t):= \frac{aS_{s}}{2} t^{3-2s} +\frac{bS_{s}^{2}}{4} t^{6-4s} -\frac{1}{2^{*}_{s}} t^{3}. 
\end{align*}
Let us note that for $t>0$,
\begin{align*}
K'(t)&= \frac{3-2s}{2} aS_{*} t^{2-2s} + \frac{3-2s}{2} bS_{*}^{2} t^{5-4s} - \frac{3-2s}{2} t^{2} \\
&= \frac{(3-2s)t^{2-2s}}{2} \left( aS_{*} + bS_{*}^{2} t^{3-2s} -t^{2s}\right) =: \frac{(3-2s)t^{2-2s}}{2} \tilde{K}(t). 
\end{align*}
Moreover, 
\begin{align*}
\tilde{K}'(t)= bS_{*} (3-2s)t^{2-2s} - 2s t^{2s-1}= t^{2-2s} [bS_{*}^{2}(3-2s) - 2s t^{4s-3}]. 
\end{align*}
Since $4s>3$, there exists a unique $T>0$ such that $\tilde{K}(t)>0$ for $t\in (0, T)$ and $\tilde{K}(t)<0$ for $t>T$. Thus, $T$ is the unique maximum point of $K(t)$. In virtue of \eqref{sq6} we have 
\begin{align}\label{sq7}
c_{\e}\leq K(T)+ O(h^{3-2s}) - O(C_{0} h^{3-\frac{(3-2s)q}{2}}). 
\end{align}
If $q>\frac{4s}{3-2s}$, then $0<3-\frac{(3-2s)q}{2}<3-2s$, and by \eqref{sq7}, for any fixed $C_{0}>0$, it holds $c_{\e}<K(T)$ for $h>0$ small. 
If $2<q<\frac{4s}{3-2s}$, then, for $h>0$ small and $C_{0}>h^{\frac{(3-2s)q}{2}-2s-1}$, we also have $c_{\e}<K(T)$. 
\end{proof}

\begin{lem}\label{lem2.4}
Every sequence $\{u_{n}\}_{n\in\mathbb{N}}$ satisfying \eqref{PSu} is bounded in $\mathcal{H}_{\e}$.
\end{lem}

\begin{proof}
In view of $(g_3)$ we can deduce that 
\begin{align}\label{cancan}
c_{\e}+ o_{n}(1)\|u_{n}\|_{\e}&\geq \mathcal{J}_{\e}(u_{n})- \frac{1}{\vartheta} \langle \mathcal{J}_{\e}'(u_{n}), u_{n}\rangle\\
&= \left(\frac{\vartheta-2}{2\vartheta}  \right)\|u_{n}\|^{2}_{\e}+b \left(\frac{\vartheta-4}{4\vartheta}  \right)[u_{n}]^{4}+\frac{1}{\vartheta}\int_{\R^{3}\setminus \Lambda_{\e}} [g(\e x, u_{n})u_{n}- \vartheta G(\e x, u_{n})]\, dx \nonumber \\
&\quad + \frac{1}{\vartheta}\int_{\Lambda_{\e}} [g(\e x, u_{n})u_{n}- \vartheta G(\e x, u_{n})]\, dx \nonumber\\
&\geq  \left(\frac{\vartheta-2}{2\vartheta}  \right)\|u_{n}\|^{2}_{\e}+\frac{1}{\vartheta}\int_{\R^{3}\setminus \Lambda_{\e}} [g(\e x, u_{n})u_{n}- \vartheta G(\e x, u_{n})]\, dx \nonumber\\
&\geq  \left(\frac{\vartheta-2}{2\vartheta}  \right)\|u_{n}\|^{2}_{\e} -\left(\frac{\vartheta-2}{2\vartheta}\right) \frac{1}{K}\int_{\R^{3}\setminus \Lambda_{\e}} V(\e x) u^{2}_{n}dx\nonumber\\
& \geq  \left(\frac{\vartheta-2}{2\vartheta}  \right)\left(1- \frac{1}{K}\right)\|u_{n}\|^{2}_{\e}.
\end{align}
Since $\vartheta >4$ and $K>2$, we can conclude that $\{u_{n}\}_{n\in\mathbb{N}}$ is bounded in $\mathcal{H}_{\e}$.
\end{proof}

\begin{lem}\label{lem445}
There exist a sequence $\{z_{n}\}_{n\in \mathbb{N}}\subset \R^{3}$ and $R, \beta>0$ such that 
\begin{align*}
\int_{B_{R}(z_{n})} u_{n}^{2} \, dx \geq \beta. 
\end{align*}
Moreover, $\{z_{n}\}_{n\in \mathbb{N}}$ is bounded in $\R^{3}$. 
\end{lem}

\begin{proof}
Assume by contradiction that the first conclusion of lemma is not true. From Lemma \ref{Lions} we have 
\begin{align*}
u_{n}\rightarrow 0 \mbox{ in } L^{q}(\R^{3}) \quad \forall q\in (2, 2^{*}_{s}), 
\end{align*}
which together with $(f_{1})$ and $(f_{2})$ yields
\begin{align*}
\int_{\R^{3}} F(u_{n})\, dx = \int_{\R^{3}} f(u_{n}) u_{n}\, dx = o_{n}(1) \mbox{ as } n\rightarrow \infty. 
\end{align*}
Since $\{u_{n}\}_{n\in \mathbb{N}}$ is bounded in $\h$, we may assume that $u_{n}\rightharpoonup u$ in $\h$.  

Now, we can observe that
\begin{align}\label{ss3}
\int_{\R^{3}} G(\e x, u_{n})\, dx \leq \frac{1}{2^{*}_{s}} \int_{\Lambda_{\e} \cup \{u_{n}\leq a_{0}\}} (u_{n}^{+})^{2^{*}_{s}} \, dx + \frac{V_{1}}{2K} \int_{(\R^{3}\setminus \Lambda_{\e}) \cap \{u_{n}>a_{0}\}} u_{n}^{2} \, dx +o_{n}(1)
\end{align}
and 
\begin{align}\label{ss4}
\int_{\R^{3}} g(\e x, u_{n}) u_{n}\, dx = \int_{\Lambda_{\e} \cup \{u_{n}\leq a_{0}\}} (u_{n}^{+})^{2^{*}_{s}} \, dx+\frac{V_{1}}{K} \int_{(\R^{3}\setminus \Lambda_{\e}) \cap \{u_{n}>a_{0}\}} u_{n}^{2} \, dx +o_{n}(1). 
\end{align}
Using $\langle \J'_{\e}(u_{n}), u_{n}\rangle =o_{n}(1)$ and \eqref{ss4} we have 
\begin{align}\label{n1}
\|u_{n}\|_{\e}^{2} - \frac{V_{1}}{K} \int_{(\R^{3}\setminus \Lambda_{\e}) \cap \{u_{n}>a_{0}\}} u_{n}^{2} \, dx+ b[u_{n}]^{4}_{s} = \int_{\Lambda_{\e} \cup \{u_{n}\leq a_{0}\}} (u_{n}^{+})^{2^{*}_{s}}\, dx + o_{n}(1). 
\end{align}
Assume that 
\begin{align*}
\int_{\Lambda_{\e} \cup \{u_{n}\leq a_{0}\}} (u_{n}^{+})^{2^{*}_{s}}\, dx \rightarrow \ell^{3}\geq 0 
\end{align*}
and 
\begin{align*}
[u_{n}]^{2}_{s}\rightarrow B^{2}. 
\end{align*}
Note that $\ell>0$, otherwise \eqref{n1} yields $\|u_{n}\|_{\e}\rightarrow 0$ as $n\rightarrow \infty$ which implies that $\J_{\e}(u_{n})\ri 0$, and this is impossible because $c_{\e}>0$. Then, by \eqref{n1} and the Sobolev inequality we obtain
\begin{align}\label{n2}
aS_{*} \left( \int_{\Lambda_{\e} \cup \{u_{n}\leq a_{0}\}}(u_{n}^{+})^{2^{*}_{s}}\, dx \right)^{\frac{2}{2^{*}_{s}}} + bS_{*}^{2} \left( \int_{\Lambda_{\e} \cup \{u_{n}\leq a_{0}\}} (u_{n}^{+})^{2^{*}_{s}}\, dx \right)^{\frac{4}{2^{*}_{s}}} \leq \int_{\Lambda_{\e} \cup \{u_{n}\leq a_{0}\}}(u_{n}^{+})^{2^{*}_{s}}\, dx +o_{n}(1). 
\end{align}
Since $\ell>0$, it follows from \eqref{n2} that 
\begin{align*}
K'(\ell)=\frac{3-2s}{2} \ell^{-1} (aS_{*} \ell^{3-2s} + bS_{*}^{2} \ell^{6-4s} -\ell^{3})\leq 0
\end{align*}
so we can deduce that $\ell\geq T$, where $T$ is the unique maximum of $K$ defined in Lemma \ref{lem444}.

Let us consider the following functional:
\begin{align}\label{ss1}
\I_{\e}(u)&:=\frac{(a+bB^{2})}{2}[u]^{2}_{s}+\frac{1}{2}\int_{\R^{3}} V(\e x) u^{2}\,dx-\int_{\R^{3}} G(\e x, u)\, dx \nonumber \\
&=\J_{\e}(u)- \frac{b}{4}[u]^{4}_{s}+ \frac{b}{2} B^{2}[u]^{2}_{s},
\end{align}
and we note that $\{u_{n}\}_{n\in \mathbb{N}}$ is a $(PS)_{c_{\e} + \frac{b}{4} B^{4}}$ sequence for $\I_{\e}$, that is 
\begin{align}\label{ss2}
\I_{\e}(u_{n})= c_{\e} + \frac{b}{4} B^{4} +o_{n}(1), \quad \I'_{\e}(u_{n})=o_{n}(1).
\end{align}
Then, using \eqref{ss3}, \eqref{ss2}, $\ell \geq T$ and the Sobolev inequality we can infer
\begin{align*}
c_{\e}&= \I_{\e}(u_{n})- \frac{b}{4}B^{4}+ o_{n}(1)\\
&\geq \frac{a}{2}[u_{n}]^{2}_{s}+\frac{bB^{2}}{2}[u_{n}]^{2}_{s}- \frac{b}{4}B^{4}+\frac{1}{2}\int_{\R^{3}} V(\e x) u_{n}^{2}dx- \frac{V_{1}}{2K} \int_{(\R^{3}\setminus \Lambda_{\e}) \cap \{u_{n}>a_{0}\}} u_{n}^{2} \, dx \\
&-\frac{1}{2^{*}_{s}} \int_{\Lambda_{\e} \cup \{u_{n}\leq a_{0}\}} (u_{n}^{+})^{2^{*}_{s}} \, dx+o_{n}(1)\\
&\geq \frac{a}{2}[u_{n}]^{2}_{s}+\frac{b}{4}[u_{n}]^{4}_{s}-\frac{1}{2^{*}_{s}} \int_{\Lambda_{\e} \cup \{u_{n}\leq a_{0}\}} (u_{n}^{+})^{2^{*}_{s}} \, dx +o_{n}(1) \\
&\geq \frac{a}{2}S_{*} \left( \int_{\Lambda_{\e} \cup \{u_{n}\leq a_{0}\}} (u_{n}^{+})^{2^{*}_{s}} \, dx \right)^{\frac{2}{2^{*}_{s}}}+\frac{b}{4}S_{*}^{2} \left( \int_{\Lambda_{\e} \cup \{u_{n}\leq a_{0}\}} (u_{n}^{+})^{2^{*}_{s}} \, dx \right)^{\frac{4}{2^{*}_{s}}}-\frac{1}{2^{*}_{s}} \int_{\Lambda_{\e} \cup \{u_{n}\leq a_{0}\}} (u_{n}^{+})^{2^{*}_{s}} \, dx +o_{n}(1) \\
&=\frac{a}{2}S_{*}\ell^{3-2s} + \frac{b}{4}S_{*}^{2}\ell^{6-4s} -\frac{1}{2^{*}_{s}} \ell^{3} \\
&\geq\frac{a}{2}S_{*}T^{3-2s} + \frac{b}{4}S_{*}^{2}T^{6-4s} -\frac{1}{2^{*}_{s}} T^{3}= c_{*}, 
\end{align*}
and this gives a contradiction by Lemma \ref{lem444}. 

Now, we show that $\{z_{n}\}_{n\in \mathbb{N}}$ is bounded in $\R^{3}$.
For any $\rho>0$, let $\psi_{\rho}\in C^{\infty}(\R^{3})$ be such that $\psi_{\rho}=0$ in $\B_{\rho}$ and $\psi_{\rho}=1$ in $\R^{3}\setminus\B_{2\rho}$, with $0\leq \psi_{\rho}\leq 1$ and $|\nabla \psi_{\rho}|\leq \frac{C}{\rho}$, where $C$ is a constant independent of $\rho$. 
Since $\{\psi_{\rho}u_{n}\}_{n\in \mathbb{N}}$ is bounded in $\mathcal{H}_{\e}$, it follows that $\langle \mathcal{J}_{\e}'(u_{n}), \psi_{\rho}u_{n}\rangle =o_{n}(1)$, that is
\begin{align*}
&(a+b [u_{n}]^{2}_{s})\iint_{\R^{6}} \frac{|u_{n}(x)- u_{n}(y)|^{2}}{|x-y|^{3+2s}} \psi_{\rho}(x) \,dxdy + \int_{\R^{3}} V(\e x) u_{n}^{2} \psi_{\rho} \, dx \\
&=o_{n}(1) + \int_{\R^{3}} g(\e x, u_{n}) u_{n}\psi_{\rho} \, dx - (a+b[u_{n}]^{2}_{s})\iint_{\R^{6}} \frac{(\psi_{\rho}(x)- \psi_{\rho}(y))(u_{n}(x)- u_{n}(y))}{|x-y|^{3+2s}} u_{n}(y) \,dxdy. 
\end{align*}
Take $\rho>0$ such that $\Lambda_{\e}\subset \B_{\rho}$. Then, using $(g_{3})$-$(ii)$, we get
\begin{align*}
&\iint_{\R^{6}} a \, \frac{|u_{n}(x)- u_{n}(y)|^{2}}{|x-y|^{3+2s}} \psi_{\rho}(x) \,dxdy + \int_{\R^{3}} V(\e x) u_{n}^{2} \psi_{\rho} \, dx\\
&\leq \int_{\R^{3}} \frac{1}{K} V(\e x) u_{n}^{2} \psi_{\rho}\, dx - (a+b[u_{n}]^{2}_{s}) \iint_{\R^{6}} \frac{(\psi_{\rho}(x)- \psi_{\rho}(y))(u_{n}(x)- u_{n}(y))}{|x-y|^{3+2s}} u_{n}(y) \,dxdy + o_{n}(1)
\end{align*}
which implies that 
\begin{align}\label{ter311}
&\left(1- \frac{1}{K}\right) V_{1} \int_{\R^{3}} u_{n}^{2} \psi_{\rho}\, dx \nonumber \\
&\leq - (a+b[u_{n}]^{2}_{s})\iint_{\R^{6}} \frac{(\psi_{\rho}(x)- \psi_{\rho}(y))(u_{n}(x)- u_{n}(y))}{|x-y|^{3+2s}} u_{n}(y) \,dxdy + o_{n}(1).
\end{align}
Now, from the H\"older inequality and the boundedness on $\{u_{n}\}_{n\in \mathbb{N}}$ in $\h$ we can see that
\begin{align}\label{ETAR}
&\left| \iint_{\R^{6}} \frac{(u_{n}(x)- u_{n}(y)) (\psi_{\rho}(x)- \psi_{\rho}(y))}{|x-y|^{3+2s}}u_{n}(y)\, dxdy\right| \nonumber \\
&\leq C \left( \iint_{\R^{6}} \frac{|\psi_{\rho}(x)- \psi_{\rho}(y)|^{2}}{|x-y|^{3+2s}}|u_{n}(y)|^{2}\, dxdy\right)^{\frac{1}{2}}. 
\end{align}

On the other hand, recalling that $0\leq \psi_{\rho}\leq 1$ and $|\nabla \psi_{\rho}|_{\infty}\leq C/\rho$ and using polar coordinates, we obtain
\begin{align*}
&\iint_{\R^{6}} \frac{|\psi_{\rho}(x)-\psi_{\rho}(y)|^{2}}{|x-y|^{3+2s}}|u_{n}(x)|^{2} dx dy \\
&=\int_{\R^{3}} \int_{|y-x|>\rho} \frac{|\psi_{\rho}(x)-\psi_{\rho}(y)|^{2}}{|x-y|^{3+2s}}|u_{n}(x)|^{2} dx dy +\int_{\R^{3}} \int_{|y-x|\leq \rho} \frac{|\psi_{\rho}(x)-\psi_{\rho}(y)|^{2}}{|x-y|^{3+2s}}|u_{n}(x)|^{2} dx dy \\
&\leq C \int_{\R^{3}} |u_{n}(x)|^{2}  \left(\int_{|y-x|>\rho} \frac{dy}{|x-y|^{3+2s}}\right) dx + \frac{C}{\rho^{2}} \int_{\R^{3}} |u_{n}(x)|^{2} \left(\int_{|y-x|\leq \rho} \frac{dy}{|x-y|^{3+2s-2}}\right) dx  \\
&\leq C \int_{\R^{3}} |u_{n}(x)|^{2}  \left(\int_{|z|>\rho} \frac{dz}{|z|^{3+2s}}\right) dx + \frac{C}{\rho^{2}} \int_{\R^{3}} |u_{n}(x)|^{2} \left(\int_{|z|\leq \rho} \frac{dz}{|z|^{1+2s}}\right) dx  \\
&\leq C \int_{\R^{3}} |u_{n}(x)|^{2} dx \left(\int_{\rho}^{\infty} \frac{d\rho}{\rho^{2s+1}}\right)  + \frac{C}{\rho^{2}} \int_{\R^{3}} |u_{n}(x)|^{2} dx \left(\int_{0}^{\rho} \frac{d\rho}{\rho^{2s-1}}\right)  \\
&\leq  \frac{C}{\rho^{2s}} \int_{\R^{3}} |u_{n}(x)|^{2} dx+\frac{C}{\rho^{2}} \rho^{-2s+2}\int_{\R^{3}} |u_{n}(x)|^{2} dx \\
&\leq  \frac{C}{\rho^{2s}} \int_{\R^{3}} |u_{n}(x)|^{2} dx\leq \frac{C}{\rho^{2s}}
\end{align*}
where in the last passage we used the boundedness of $\{u_{n}\}_{n\in \mathbb{N}}$ in $\h$. 
Taking into account \eqref{ter311}, \eqref{ETAR} and the above estimate we can infer that
$$
\left(1- \frac{1}{K}\right) V_{1} \int_{\R^{3}} u_{n}^{2} \psi_{\rho}\, dx \leq \frac{C}{\rho^{s}}+o_{n}(1)
$$
which implies that $\{z_{n}\}_{n\in \mathbb{N}}$ is bounded in $\R^{3}$. 
\end{proof}

\noindent
We conclude this section giving the proof of the main result of this section:
\begin{thm}\label{GSfcK}
Assume that $(V_1)$-$(V_2)$ and $(f1)$-$(f4)$ hold. Then, problem \eqref{MP} admits a positive ground state for all $\e>0$.
\end{thm}
\begin{proof}
Using Lemma \ref{lem2.2} and a variant of the mountain pass theorem without $(PS)$ condition (see \cite{Willem}), we know that there exists a Palais-Smale sequence $\{u_{n}\}_{n\in \mathbb{N}}$ for $\J_{\e}$ at the level $c_{\e}$, where $c_{\e}<c_{*}$ by Lemma \ref{lem444}.
Taking into account Lemma \ref{lem2.4}, we can see that $\{u_{n}\}_{n\in \mathbb{N}}$ is bounded in $\h$, so we may assume that $u_{n}\rightharpoonup u$ in $\h$ and $u_{n}\rightarrow u$ in $L^{q}_{loc}(\R^{3})$ for all $q\in [1, \p)$.
It follows from Lemma \ref{lem445} that $u$ nontrivial. 
Since $\langle \J'_{\e}(u_{n}), \varphi\rangle=o_{n}(1)$ for all $\varphi\in \h$, we can see that
\begin{align}\label{HZman}
 \int_{\R^{3}} a(-\Delta)^{\frac{s}{2}} u (-\Delta)^{\frac{s}{2}}\varphi+V(\e x) u \varphi \, dx+& bB^{2} \left( \int_{\R^{3}} (-\Delta)^{\frac{s}{2}} u (-\Delta)^{\frac{s}{2}}\varphi \, dx\right)=\int_{\R^{3}} g(\e x, u)\varphi \, dx,
\end{align}
where $B^{2}:=\lim_{n\ri \infty} [u_{n}]^{2}_{s}$.
Let us note that $B^{2}\geq [u]^{2}_{s}$ by Fatou's Lemma. If by contradiction $B^{2}>[u]^{2}_{s}$, we may use \eqref{HZman} to deduce that $\langle \J'_{\e}(u), u\rangle<0$. Moreover, conditions $(g_1)$-$(g_2)$ imply that $\langle \J'_{\e}(\tau u), \tau u\rangle>0$ for some $0<\tau<<1$. Then there exists $t_{0}\in (\tau, 1)$ such that $t_{0} u\in \N_{\e}$ and $\langle \J'_{\e}(t_{0} u), t_{0} u\rangle=0$. Using Fatou's Lemma, $t_{0}\in (\tau, 1)$ and $(g_3)$ we get
\begin{align}\label{3.11HLP}
c_{\e}&\leq \J_{\e}(t_{0} u)-\frac{1}{4} \langle \J'_{\e}(t_{0} u), t_{0} u\rangle <\J_{\e}(u)-\frac{1}{4} \langle \J'_{\e}(u), u\rangle \leq \liminf_{n\rightarrow \infty} \left[\J_{\e}(u_{n})-\frac{1}{4} \langle \J'_{\e}(u_{n}), u_{n}\rangle\right]=c_{\e}
\end{align}
which gives a contradiction. Therefore $B^{2}= [u]^{2}_{s}$ and we deduce that $\J'_{\e}(u)=0$.
Hence, $\J_{\e}$ admits a nontrivial critical point $u\in \h$. 
Since $\langle \J_{\e}'(u), u^{-}\rangle=0$, where $u^{-}=\min\{u,0\}$, and $g(x,t)=0$ for $t\leq 0$, it is easy to check that $u\geq 0$ in $\R^{3}$. Moreover, proceeding as in the proof of Lemma \ref{Moser} below, we can see that $u\in L^{\infty}(\R^{3})$. By Proposition $2.9$ in \cite{S} and $s>\frac{3}{4}$ we deduce that $u\in C^{1, \alpha}(\R^{3})$, and applying the maximum principle \cite{S} we can conclude that $u>0$ in $\R^{3}$.
Finally, arguing as in \eqref{3.11HLP} with $t_{0}=1$, we can show that $u$ is a ground state solution to \eqref{MP}. 
\end{proof}

\subsection{The limiting problem}\hfill\\
\noindent
Let us consider the following limiting problem related to \eqref{MP}, that is, for $\mu>0$
\begin{equation}\label{P0}
\left\{
\begin{array}{ll}
(a+b[u]^{2}_{s})(-\Delta)^{s}u +  \mu u= f(u)+|u|^{2^{*}_{s}-2}u \, &\mbox{ in } \R^{3}, \\
u\in H^{s}(\R^{3}), \quad u>0 &\mbox{ in } \R^{3},
\end{array}
\right. 
\end{equation}
whose corresponding Euler-Lagrange functional is given by
\begin{equation*}
\mathcal{I}_{\mu}(u)= \frac{1}{2}\left(a[u]_{s}^{2} +\mu |u|_{2}^{2}\right)+\frac{b}{4}[u]^{4}_{s}- \int_{\R^{3}} F(u)+\frac{1}{2^{*}_{s}} (u^{+})^{2^{*}_{s}} \,dx
\end{equation*}
which is well defined on the Hilbert space $\mathcal{H}_{\mu}:=H^{s}(\R^{3})$ endowed with the inner product
\begin{equation*}
(u, \varphi)_{\mu} := a\iint_{\R^{6}} \frac{(u(x)- u(y))(\varphi(x) - \varphi(y))}{|x-y|^{3+2s}} dxdy +\mu \int_{\R^{3}} u(x)\, \varphi(x) dx. 
\end{equation*}
The norm induced by the above inner product is given by
\begin{equation*}
\|u\|_{\mu}^{2} :=a [u]^{2}_{s} + \mu |u|_{2}^{2}. 
\end{equation*}
We denote by $\mathcal{M}_{\mu}$ the Nehari manifold associated with $\mathcal{I}_{\mu}$, that is
\begin{equation*}
\mathcal{M}_{\mu}:= \Bigl\{u\in \mathcal{H}_{\mu}\setminus \{0\} : \langle \mathcal{I}_{\mu}'(u), u \rangle=0 \Bigr\}, 
\end{equation*}
and 
$$
d_{\mu}:=\inf_{u\in \mathcal{M}_{\mu}} \I_{\mu}(u),
$$
or equivalently
$$
d_{\mu}=\inf_{u\in \mathcal{H}_{\mu}\setminus \{0\}} \max_{t\geq 0} \I_{\mu}(t u).
$$

\noindent
Arguing as in the proof of Theorem \ref{GSfcK}, it is easy to deduce that:
\begin{thm}\label{thm3.1}
For all $\mu>0$, problem \eqref{P0} admits a positive ground state solution.
\end{thm}

Let us prove the following useful relation between $c_{\e}$ and $d_{V_{0}}$:
\begin{lem}\label{lemcec0}
It holds $\limsup_{\e\rightarrow 0} c_{\e}\leq d_{V_{0}}$.
\end{lem}
\begin{proof}
For any $\e>0$ we set $\omega_{\e}(x):=\psi_{\e}(x)\omega(x)$, where $\omega$ is a positive ground state given by Theorem \ref{thm3.1} with $\mu=V_{0}$, and $\psi_{\e}(x):=\psi(\e x)$ with $\psi\in C^{\infty}_{c}(\R^{3})$, $\psi\in [0,1]$, $\psi(x)=1$ if $|x|\leq \frac{1}{2}$ and $\psi(x)=0$ if $|x|\geq 1$. Here we assume that $supp(\psi)\subset \B_{1}\subset \Lambda$. Using Lemma \ref{pp} and the dominated convergence theorem we can see that 
$\omega_{\e}\rightarrow \omega$ in  $H^{s}(\R^{3})$ and  $\I_{V_{0}}(\omega_{\e})\ri \I_{V_{0}}(\omega)=d_{V_{0}}$ as  $\e\rightarrow 0$.
For each $\e>0$ there exists $t_{\e}>0$ such that
$$
\J_{\e}(t_{\e} \omega_{\e})=\max_{t\geq 0} \J_{\e}(t \omega_{\e}).
$$
Then, $\J_{\e}'(t_{\e} \omega_{\e})=0$ and this implies that
\begin{align}\label{LZ1}
&\frac{1}{t_{\e}^{2}} \int_{\R^{3}} a|(-\Delta)^{\frac{s}{2}}\omega_{\e}|^{2}+V(\e x)\omega_{\e}^{2} \, dx+b \left(\int_{\R^{3}} |(-\Delta)^{\frac{s}{2}}\omega_{\e}|^{2}dx\right)^{2} \nonumber \\
&=\int_{\R^{3}} \frac{f(t_{\e}\omega_{\e})}{(t_{\e}\omega_{\e})^{3}}\omega_{\e}^{4}\, dx+t_{\e}^{2^{*}_{s}-4}\int_{\R^{3}} |\omega_{\e}|^{2^{*}_{s}}\, dx.
\end{align}
By $(f_1)$-$(f_4)$, $\omega\in \mathcal{M}_{V_{0}}$ and \eqref{LZ1} it follows that $t_{\e}\rightarrow 1$ as $\e\ri 0$. On the other hand,
\begin{align*}
c_{\e}\leq \max_{t\geq 0} \J_{\e}(t \omega_{\e})&=\J_{\e}(t_{\e}\omega_{\e})=\I_{V_{0}}(t_{\e}\omega_{\e})+\frac{t^{2}_{\e}}{2} \int_{\R^{3}} (V(\e x)-V_{0}) \omega^{2}_{\e}\, dx.
\end{align*}
Since $V(\e x)$ is bounded on the support of $\omega_{\e}$, by the dominated convergence theorem and the above inequality, we obtain the thesis.
\end{proof}

\section{Proof of Theorem \ref{thm1}}
\noindent
This last section is devoted to the proof of the main result of this work.
Firstly, we prove the following compactness result which will be fundamental to show that the solutions of \eqref{MP} are also solutions to \eqref{Pe} for $\e>0$ small enough.
\begin{lem}\label{lem3.1N}
Let $\e_{n}\rightarrow 0^{+}$ and $\{u_{n}\}_{n\in \mathbb{N}}:=\{u_{\e_{n}}\}_{n\in \mathbb{N}}\subset \mathcal{H}_{\e_{n}}$ be such that $\mathcal{J}_{\e_{n}}(u_{n})= c_{\e_{n}}$ and $\mathcal{J}'_{\e_{n}}(u_{n})=0$. Then there exists $\{\tilde{y}_{n}\}_{n\in \mathbb{N}}\subset \R^{3}$ such that the translated sequence 
\begin{equation*}
\tilde{u}_{n}(x):=u_{n}(x+ \tilde{y}_{n})
\end{equation*}
has a subsequence which converges in $H^{s}(\R^{3})$. Moreover, up to a subsequence, $\{y_{n}\}_{n\in \mathbb{N}}:=\{\e_{n}\tilde{y}_{n}\}_{n\in \mathbb{N}}$ is such that $y_{n}\rightarrow y_{0}$ for some $y_{0}\in \Lambda$ such that $V(y_{0})=V_{0}$. 
\end{lem}

\begin{proof}
Using $\langle \mathcal{J}'_{\e_{n}}(u_{n}), u_{n} \rangle=0$ and $(g_1)$, $(g_2)$, it is easy to see that 
there is $\gamma>0$ (independent of $\e_{n}$) such that
$$
\|u_{n}\|_{\e_{n}}\geq \gamma>0 \quad \forall n\in \mathbb{N}.
$$
Taking into account $\mathcal{J}_{\e_{n}}(u_{n})= c_{\e_{n}}$, $\langle \mathcal{J}'_{\e_{n}}(u_{n}), u_{n}\rangle=0$ and Lemma \ref{lemcec0}, we can argue as in the proof of Lemma \ref{lem2.4} to deduce that $\{u_{n}\}_{n\in \mathbb{N}}$ is bounded in $\mathcal{H}_{\e_{n}}$. 
Therefore, proceeding as in Lemma \ref{lem445}, we can find a sequence $\{\tilde{y}_{n}\}_{n\in \mathbb{N}}\subset \R^{3}$ and constants $R, \alpha>0$ such that
\begin{equation*}
\liminf_{n\rightarrow \infty}\int_{\B_{R}(\tilde{y}_{n})} |u_{n}|^{2} dx\geq \alpha.
\end{equation*}
Set $\tilde{u}_{n}(x):=u_{n}(x+ \tilde{y}_{n})$. Then, $\{\tilde{u}_{n}\}_{n\in \mathbb{N}}$ is bounded in $H^{s}(\R^{3})$, and we may assume that 
\begin{equation}\label{WTILDE}
\tilde{u}_{n}\rightharpoonup \tilde{u} \mbox{ weakly in } H^{s}(\R^{3}),
\end{equation}
and $[\tilde{u}_{n}]^{2}_{s}\ri B^{2}$ as $n\rightarrow \infty$.
Moreover,  $\tilde{u}\neq 0$ in view of
\begin{align}\label{ALPHATILDEU}
\int_{\B_{R}} |\tilde{u}|^{2} dx\geq \alpha.
\end{align}
Now, we set $y_{n}:=\e_{n}\tilde{y}_{n}$. Firstly, we show that $\{y_{n}\}_{n\in \mathbb{N}}$ is bounded.
To achieve our purpose, we prove the following claim: \\
{\bf Claim 1} $\lim_{n\rightarrow \infty} dist(y_{n}, \overline{\Lambda})=0$. \\
If by contradiction the claim is not true, then we can find $\delta>0$ and a subsequence of $\{y_{n}\}_{n\in \mathbb{N}}$, still denoted by itself, such that
$$
dist(y_{n}, \overline{\Lambda})\geq \delta \quad \forall n\in \mathbb{N}.
$$
Thus, there is $r>0$ such that $\B_{r}(y_{n})\subset \R^{3}\setminus \Lambda$ for all $n\in \mathbb{N}$. Since $\tilde{u}\geq 0$ and $C^{\infty}_{c}(\R^{3})$ is dense in $H^{s}(\R^{3})$, we can approximate $\tilde{u}$ by a sequence $\{\psi_{j}\}_{j\in \mathbb{N}}\subset C^{\infty}_{c}(\R^{3})$ such that $\psi_{j}\geq 0$ in $\R^{3}$, so that $\psi_{j}\rightarrow  \tilde{u}$ in $H^{s}(\R^{3})$. Fix $j\in \mathbb{N}$ and use $\psi=\psi_{j}$ as test function in $\langle \J'_{\e_{n}}(u_{n}), \psi\rangle=0$. Then we have
\begin{align}\label{2.17AM}
&(a+b[\tilde{u}_{n}]^{2}_{s})\iint_{\R^{6}} \frac{(\tilde{u}_{n}(x)-\tilde{u}_{n}(y))(\psi_{j}(x)-\psi_{j}(y))}{|x-y|^{3+2s}} \,dx dy+ \int_{\R^{3}} V(\e_{n} x+\e_{n} \tilde{y}_{n}) \tilde{u}_{n}  \psi_{j}\, dx
 \nonumber \\
&=\int_{\R^{3}} g(\e_{n} x+\e_{n}\tilde{y}_{n}, \tilde{u}_{n})\psi_{j} \,dx .
\end{align}
Since $u_{\e_{n}}, \psi_{j}\geq 0$ and using the definition of $g$,  we can note that
\begin{align*}
\int_{\R^{3}} g(\e_{n} x+\e_{n}\tilde{y}_{n}, \tilde{u}_{n})\psi_{j} \,dx&=  \int_{\B_{r/\e_{n}}} g(\e_{n} x+\e_{n}\tilde{y}_{n}, \tilde{u}_{n})\psi_{j} \,dx + \int_{\R^{3}\setminus \B_{r/\e_{n}}} g(\e_{n} x+\e_{n}\tilde{y}_{n}, \tilde{u}_{n})\psi_{j} \,dx \nonumber \\
&\leq \frac{V_{1}}{K} \int_{\B_{r/\e_{n}}} \tilde{u}_{n}\psi_{j} \, dx+\int_{\R^{3}\setminus \B_{r/\e_{n}}} \left(f(\tilde{u}_{n})\psi_{j} +\tilde{u}_{n}^{\p-1} \psi_{j}\right)\, dx. 
\end{align*}
This fact together with \eqref{2.17AM} gives
\begin{align}\label{2.18AM}
&(a+b[\tilde{u}_{n}]^{2}_{s}) \iint_{\R^{6}} \frac{(\tilde{u}_{n}(x)-\tilde{u}_{n}(y))(\psi_{j}(x)-\psi_{j}(y))}{|x-y|^{3+2s}} \,dx dy+ A\int_{\R^{3}} \tilde{u}_{n} \psi_{j}\, dx
 \nonumber \\
&\leq \int_{\R^{3}\setminus \B_{r/\e_{n}}} \left(f(\tilde{u}_{n})\psi_{j} +\tilde{u}_{n}^{\p-1} \psi_{j}\right)\, dx
\end{align}
where $A=V_{1}(1-\frac{1}{K})$. Taking into account \eqref{WTILDE}, $\psi_{j}$ has compact support in $\R^{3}$ and $\e_{n}\rightarrow 0^{+}$, we can infer that as $n\rightarrow \infty$
\begin{align*}
&\iint_{\R^{6}} \frac{(\tilde{u}_{n}(x)-\tilde{u}_{n}(y))(\psi_{j}(x)-\psi_{j}(y))}{|x-y|^{3+2s}} \,dx dy \\
&\quad \rightarrow \iint_{\R^{6}} \frac{(\tilde{u}(x)-\tilde{u}(y))(\psi_{j}(x)-\psi_{j}(y))}{|x-y|^{3+2s}} \,dx dy
\end{align*}
and
$$
\int_{\R^{3}\setminus \B_{r/\e_{n}}} \left(f(\tilde{u}_{n})\psi_{j} +\tilde{u}_{n}^{\p-1} \psi_{j}\right)\, dx\rightarrow 0.
$$
The above limits, \eqref{2.18AM} and $[\tilde{u}_{n}]^{2}_{s}\ri B^{2}$ imply that
$$
(a+bB^{2})\iint_{\R^{6}} \frac{(\tilde{u}(x)-\tilde{u}(y))(\psi_{j}(x)-\psi_{j}(y))}{|x-y|^{3+2s}} \,dx dy+A \int_{\R^{3}} \tilde{u} \psi_{j}\, dx\leq 0
$$
and passing to the limit as $j\rightarrow \infty$ we can infer that 
$$
(a+bB^{2})[\tilde{u}]^{2}_{s}+A|\tilde{u}|_{2}^{2}\leq 0.
$$
This gives a  contradiction by  \eqref{ALPHATILDEU}. 
Hence, there exists a subsequence of $\{y_{n}\}_{n\in \mathbb{N}}$ such that 
$$
y_{n}\rightarrow y_{0}\in \overline{\Lambda}.
$$
Secondly, we prove the following claim:\\
{\bf Claim 2} $y_{0}\in \Lambda$. \\
In the light of $(g_2)$ and \eqref{2.17AM} we can deduce that
\begin{align*}
&(a+b[\tilde{u}_{n}]^{2}_{s}) \iint_{\R^{6}} \frac{(\tilde{u}_{n}(x)-\tilde{u}_{n}(y))(\psi_{j}(x)-\psi_{j}(y))}{|x-y|^{3+2s}} \,dx dy+ \int_{\R^{3}} V(\e_{n} x+\e_{n} \tilde{y}_{n}) \tilde{u}_{n}  \psi_{j}\, dx
 \nonumber \\
&\leq \int_{\R^{3}} (f(\tilde{u}_{n})+\tilde{u}_{n}^{\p-1})\psi_{j} \,dx. 
\end{align*}
Letting $n\rightarrow \infty$ we find
\begin{align*}
&(a+bB^{2})\iint_{\R^{6}} \frac{(\tilde{u}(x)-\tilde{u}(y))(\psi_{j}(x)-\psi_{j}(y))}{|x-y|^{3+2s}} \,dx dy+ \int_{\R^{3}} V(y_{0}) \tilde{u}  \psi_{j}\, dx
 \nonumber \\
&\leq \int_{\R^{3}} (f(\tilde{u})+\tilde{u}^{\p-1})\psi_{j} \,dx,
\end{align*}
and passing to the limit as $j\rightarrow \infty$ we obtain
\begin{align*}
(a+bB^{2}) [\tilde{u}]^{2}_{s}+ V(y_{0}) |\tilde{u}|^{2}_{2} 
\leq \int_{\R^{3}} (f(\tilde{u})+\tilde{u}^{\p-1})\tilde{u} \,dx.
\end{align*}
Since $B^{2}\geq [\tilde{u}]^{2}_{s}$ (by Fatou's Lemma), the above inequality yields 
\begin{align*}
(a+b[\tilde{u}]^{2}_{s}) [\tilde{u}]^{2}_{s}+ V(y_{0}) |\tilde{u}|^{2}_{2}
\leq \int_{\R^{3}} (f(\tilde{u})+\tilde{u}^{\p-1})\tilde{u} \,dx.
\end{align*}
Therefore, we can find $\tau\in (0, 1)$ such that $\tau\tilde{u}\in \mathcal{M}_{V(y_{0})}$.
Then, by Lemma \ref{lemcec0}, we can see that
\begin{align*}
d_{V(y_{0})}\leq \mathcal{I}_{V(y_{0})}(\tau \tilde{u})\leq \liminf_{n\rightarrow \infty} \J_{\e_{n}}(u_{n})=\liminf_{n\rightarrow \infty}c_{\e_{n}}\leq d_{V_{0}}
\end{align*}
which implies that $V(y_{0})\leq V(0)=V_{0}$. Since $V_{0}=\min_{\bar{\Lambda}} V$, we can deduce that $V(y_{0})=V_{0}$. This fact together with $(V_2)$ yields $y_{0}\notin \partial \Lambda$. Consequently, $y_{0}\in \Lambda$. \\
{\bf Claim 3} $\tilde{u}_{n}\rightarrow \tilde{u}$ in  $H^{s}(\R^{3})$ as $n\rightarrow \infty$. \\
Let us define
$$
\tilde{\Lambda}_{n} := \frac{\Lambda - \e_{n}\tilde{y}_{n}}{\e_{n}}
$$ 
and 
\begin{align*}
&\tilde{\chi}_{n}^{1}(x):= \left\{
\begin{array}{ll}
1 \, &\mbox{ if } x\in \tilde{\Lambda}_{n},\\
0 \, &\mbox{ if } x\in \R^{3}\setminus \tilde{\Lambda}_{n}, 
\end{array}
\right.\\
&\tilde{\chi}_{n}^{2}(x):= 1- \tilde{\chi}_{n}^{1}(x).
\end{align*}
Let us also consider the following functions for all $x\in \R^{3}$
\begin{align*}
&h_{n}^{1}(x):= \left(\frac{1}{2}-\frac{1}{\vartheta}\right) V(\e_{n}x+ \e_{n}\tilde{y}_{n}) |\tilde{u}_{n}(x)|^{2} \tilde{\chi}_{n}^{1}(x)\\
&h^{1}(x):=\left(\frac{1}{2}-\frac{1}{\vartheta}\right) V(y_{0}) |\tilde{u}(x)|^{2} \\
&h_{n}^{2}(x)\!\!:=\!\!\left[ \left(\frac{1}{2}-\frac{1}{\vartheta}\right) V(\e_{n}x+ \e_{n}\tilde{y}_{n}) |\tilde{u}_{n}(x)|^{2} + \frac{1}{\vartheta} g(\e_{n}x+ \e_{n}\tilde{y}_{n}, \tilde{u}_{n}(x)) \tilde{u}_{n}(x) - G(\e_{n}x+ \e_{n}\tilde{y}_{n}, \tilde{u}_{n}(x))\right] \tilde{\chi}_{n}^{2}(x) \\
&\quad \quad \, \, \, \geq \left( \left(\frac{1}{2}-\frac{1}{\vartheta}\right) -\frac{1}{K}\right) V(\e_{n}x+ \e_{n}\tilde{y}_{n}) |\tilde{u}_{n}(x)|^{2} \tilde{\chi}_{n}^{2}(x) \\
&h_{n}^{3}(x):= \left(\frac{1}{\vartheta} g(\e_{n}x+ \e_{n}\tilde{y}_{n}, \tilde{u}_{n}(x)) \tilde{u}_{n}(x) - G(\e_{n}x+ \e_{n}\tilde{y}_{n}, \tilde{u}_{n}(x))\right) \tilde{\chi}_{n}^{1}(x) \\
&\quad \quad \, \, \, =\left[\frac{1}{\vartheta} \left(f(\tilde{u}_{n}(x))\tilde{u}_{n}(x) + |\tilde{u}_{n}(x)|^{\p}\right)- \left(F(\tilde{u}_{n}(x))+ \frac{1}{\p}|\tilde{u}_{n}(x)|^{\p}\right) \right] \tilde{\chi}_{n}^{1}(x)  \\
&h^{3}(x):= \frac{1}{\vartheta} \left(f(\tilde{u}(x))\tilde{u}(x) + |\tilde{u}(x)|^{\p}\right)- \left(F(\tilde{u}(x))+ \frac{1}{\p}|\tilde{u}(x)|^{\p}\right). 
\end{align*}
In view of $(f_3)$ and $(g_3)$, we can observe that the above functions are nonnegative.
Moreover, by \eqref{WTILDE} and Claim $2$, we know that 
\begin{align*}
&\tilde{u}_{n}(x) \ri \tilde{u}(x)\quad \mbox{ a.e. } x\in \R^{3}, \\
& y_{n}=\e_{n}\tilde{y}_{n}\ri y_{0}\in \Lambda,
\end{align*}
which imply that
\begin{align*}
&\tilde{\chi}_{n}^{1}(x)\ri 1, \, h_{n}^{1}(x)\ri h^{1}(x), \, h_{n}^{2}(x)\ri 0 \, \mbox{ and } \, h_{n}^{3}(x)\ri h^{3}(x) \, \mbox{ a.e. } x\in \R^{3}. 
\end{align*}
Hence, applying Fatou's Lemma and using the invariance of $\R^{3}$ by translation, we can see that
\begin{align*}
d_{V_{0}} &\geq \limsup_{n\ri \infty} c_{\e_{n}} = \limsup_{n\ri \infty} \left( \J_{\e_{n}}(u_{n}) - \frac{1}{\vartheta} \langle \J'_{\e_{n}}(u_{n}), u_{n}\rangle \right)\\
&\geq \limsup_{n\ri \infty} \left[\left(\frac{1}{2}-\frac{1}{\vartheta} \right)[\tilde{u}_{n}]^{2}_{s}+\left(\frac{1}{4}-\frac{1}{\vartheta} \right)b[\tilde{u}_{n}]^{4}_{s}+  \int_{\R^{3}} (h_{n}^{1}+ h_{n}^{2}+ h_{n}^{3}) \, dx\right]\\
&\geq \liminf_{n\ri \infty} \left[\left(\frac{1}{2}-\frac{1}{\vartheta} \right)[\tilde{u}_{n}]^{2}_{s}+\left(\frac{1}{4}-\frac{1}{\vartheta} \right) b[\tilde{u}_{n}]^{4}_{s} + \int_{\R^{3}} (h_{n}^{1}+ h_{n}^{2}+ h_{n}^{3}) \, dx \right]\\
&\geq \left(\frac{1}{2}-\frac{1}{\vartheta} \right)[\tilde{u}]^{2}_{s}+\left(\frac{1}{4}-\frac{1}{\vartheta} \right)b[\tilde{u}]^{4}_{s}+ \int_{\R^{3}} (h^{1}+ h^{3}) \, dx\geq d_{V_{0}}.
\end{align*}
Accordingly
\begin{align}\label{2.19AM}
\lim_{n\ri \infty}[\tilde{u}_{n}]^{2}_{s} = [\tilde{u}]^{2}_{s}
\end{align}
and 
\begin{align*}
h_{n}^{1}\ri h^{1}, \, h_{n}^{2}\ri 0 \, \mbox{ and }\, h_{n}^{3}\ri h^{3} \, \mbox{ in } \, L^{1}(\R^{3}). 
\end{align*}
Then
\begin{align*}
\lim_{n\ri \infty} \int_{\R^{3}} V(\e_{n} x+ \e_{n}\tilde{y}_{n})|\tilde{u}_{n}|^{2} \, dx = \int_{\R^{3}} V(y_{0})|\tilde{u}|^{2} \, dx, 
\end{align*}
and we can deduce that 
\begin{align}\label{2.20AM}
\lim_{n\ri \infty} |\tilde{u}_{n}|_{2}^{2}= |\tilde{u}|_{2}^{2}. 
\end{align}
Putting together \eqref{WTILDE}, \eqref{2.19AM} and \eqref{2.20AM} and using the fact that $H^{s}(\R^{3})$ is a Hilbert space we obtain
\begin{align*}
\|\tilde{u}_{n}- \tilde{u}\|_{V_{0}}\ri 0 \mbox{ as } n\ri \infty.
\end{align*}
This fact ends the proof of lemma. 
\end{proof}

\noindent
In the next lemma, we use a Moser iteration argument \cite{Moser} to prove the following useful $L^{\infty}$-estimate for the solutions of the modified problem \eqref{MP}. 
\begin{lem}\label{Moser}
Let $\e_{n}\rightarrow 0$ and $u_{n}\in \mathcal{H}_{\e_{n}}$ be a solution to \eqref{MP}. Then, up to a subsequence, $\tilde{u}_{n}:=u_{n}(\cdot+\tilde{y}_{n})\in L^{\infty}(\R^{3})$, and there exists $C>0$ such that 
\begin{equation*}
|\tilde{u}_{n}|_{\infty}\leq C \quad \mbox{ for all } n\in \mathbb{N}.
\end{equation*}
\end{lem}
\begin{proof}
For any $L>0$ and $\beta>1$, let us define the function 
\begin{equation*}
\gamma(\tilde{u}_{n}):=\gamma_{L, \beta}(\tilde{u}_{n})=\tilde{u}_{n} \tilde{u}_{L, n}^{2(\beta-1)}\in \h
\end{equation*}
where  $\tilde{u}_{L,n}:=\min\{\tilde{u}_{n}, L\}$. 
Since $\gamma$ is an increasing function, we have
\begin{align*}
(a-b)(\gamma(a)- \gamma(b))\geq 0 \quad \mbox{ for any } a, b\in \R.
\end{align*}
Let us consider 
\begin{equation*}
\mathcal{E}(t):=\frac{|t|^{2}}{2} \quad \mbox{ and } \quad \Gamma(t):=\int_{0}^{t} (\gamma'(\tau))^{\frac{1}{2}} d\tau. 
\end{equation*}
Then, applying Jensen's inequality we get for all $a, b\in \R$ such that $a>b$,
\begin{align*}
\mathcal{E}'(a-b)(\gamma(a)-\gamma(b)) &=(a-b) (\gamma(a)-\gamma(b))= (a-b) \int_{b}^{a} \gamma'(t) dt \\
&= (a-b) \int_{b}^{a} (\Gamma'(t))^{2} dt \geq \left(\int_{b}^{a} (\Gamma'(t)) dt\right)^{2}.
\end{align*}
The same argument works when $a\leq b$. Therefore
\begin{equation}\label{Gg}
\mathcal{E}'(a-b)(\gamma(a)-\gamma(b))\geq |\Gamma(a)-\Gamma(b)|^{2} \quad \mbox{ for any } a, b\in\R. 
\end{equation}
From \eqref{Gg}, we can see that
\begin{align}\label{Gg1}
|\Gamma(\tilde{u}_{n})(x)- \Gamma(\tilde{u}_{n})(y)|^{2} \leq (\tilde{u}_{n}(x)- \tilde{u}_{n}(y))((\tilde{u}_{n}\tilde{u}_{L,n}^{2(\beta-1)})(x)- (\tilde{u}_{n}\tilde{u}_{L,n}^{2(\beta-1)})(y)). 
\end{align}
Choosing $\gamma(\tilde{u}_{n})=\tilde{u}_{n} \tilde{u}_{L, n}^{2(\beta-1)}$ as test function in \eqref{MP} and using \eqref{Gg1} we obtain
\begin{align}\label{BMS}
&a[\Gamma(\tilde{u}_{n})]^{2}_{s}+\int_{\R^{3}} V_{n}(x)|\tilde{u}_{n}|^{2}\tilde{u}_{L, n}^{2(\beta-1)} dx\nonumber \\
&\leq (a+b[\tilde{u}_{n}]^{2}_{s})\iint_{\R^{6}} \frac{(\tilde{u}_{n}(x)- \tilde{u}_{n}(y))}{|x-y|^{N+2s}} ((\tilde{u}_{n}\tilde{u}_{L, n}^{2(\beta-1)})(x)-(\tilde{u}_{n} \tilde{u}_{L,n}^{2(\beta-1)})(y)) \,dx dy +\int_{\R^{3}} V_{n}(x)|\tilde{u}_{n}|^{2}\tilde{u}_{L,n}^{2(\beta-1)} dx \nonumber\\
&\leq \int_{\R^{3}} g_{n}(\tilde{u}_{n}) \tilde{u}_{n} \tilde{u}_{L,n}^{2(\beta-1)} dx,
\end{align}
where 
$V_{n}(x):=V(\e_{n} x+\e_{n} \tilde{y}_{n})$ and $g_{n}(x):=g(\e_{n} x+\e_{n} \tilde{y}_{n}, \tilde{u}_{n})$.  
Since 
$$
\Gamma(\tilde{u}_{n})\geq \frac{1}{\beta} \tilde{u}_{n} \tilde{u}_{L,n}^{\beta-1},
$$ 
and by Theorem \ref{Sembedding}, we have
\begin{equation}\label{SS1}
[\Gamma(\tilde{u}_{n})]^{2}_{s}\geq S_{*} |\Gamma(\tilde{u}_{n})|^{2}_{2^{*}_{s}}\geq \left(\frac{1}{\beta}\right)^{2} S_{*}|\tilde{u}_{n} \tilde{u}_{L,n}^{\beta-1}|^{2}_{2^{*}_{s}}.
\end{equation}
On the other hand, by assumptions $(g_1)$ and $(g_2)$, for any $\xi>0$ there exists $C_{\xi}>0$ such that
\begin{equation}\label{SS2}
|g_{n}(\tilde{u}_{n})|\leq \xi |\tilde{u}_{n}|+C_{\xi}|\tilde{u}_{n}|^{2^{*}_{s}-1}.
\end{equation}
Thus, taking $\xi\in (0, V_{1})$, and from \eqref{SS1} and \eqref{SS2}, we can see that \eqref{BMS} yields
\begin{align}\label{simo1}
|w_{L,n}|^{2}_{\p}&\leq C\beta^{2} \int_{\R^{3}} |\tilde{u}_{n}|^{\p} \tilde{u}_{L,n}^{2(\beta-1)} dx. 
\end{align}
where $w_{L,n}:=\tilde{u}_{n} \tilde{u}_{L,n}^{\beta-1}$. 
Now, we take $\beta=\frac{\p}{2}$ and fix $R>0$. Recalling that $0\leq \tilde{u}_{L,n}\leq \tilde{u}_{n}$, we have
\begin{align}\label{simo2}
\int_{\R^{3}} \tilde{u}^{\p}_{n}v_{L,n}^{2(\beta-1)}dx&=\int_{\R^{3}} \tilde{u}^{\p-2}_{n} \tilde{u}^{2}_{n} v_{L,n}^{\p-2}dx \nonumber\\
&=\int_{\R^{3}} \tilde{u}^{\p-2}_{n} (\tilde{u}_{n} \tilde{u}_{L,n}^{\frac{\p-2}{2}})^{2}dx \nonumber\\
&\leq \int_{\{\tilde{u}_{n}<R\}} R^{\p-2} \tilde{u}^{\p}_{n} dx+\int_{\{\tilde{u}_{n}>R\}} \tilde{u}^{\p-2}_{n} (\tilde{u}_{n} \tilde{u}_{L,n}^{\frac{\p-2}{2}})^{2}dx \nonumber\\
&\leq \int_{\{\tilde{u}_{n}<R\}} R^{\p-2} \tilde{u}^{\p}_{n} dx+\left(\int_{\{\tilde{u}_{n}>R\}} \tilde{u}^{\p}_{n} dx\right)^{\frac{\p-2}{\p}} \left(\int_{\R^{3}} (\tilde{u}_{n} \tilde{u}_{L,n}^{\frac{\p-2}{2}})^{\p}dx\right)^{\frac{2}{\p}}.
\end{align}
Since $\{\tilde{u}_{n}\}_{n\in \mathbb{N}}$ is bounded in $L^{\p}(\R^{3})$, we can see that for any $R$ sufficiently large
\begin{equation}\label{simo3}
\left(\int_{\{\tilde{u}_{n}>R\}} \tilde{u}^{\p}_{n} dx\right)^{\frac{\p-2}{\p}}\leq \frac{1}{2C \beta^{2}}.
\end{equation}
Putting together \eqref{simo1}, \eqref{simo2} and \eqref{simo3} we get
\begin{equation*}
\left(\int_{\R^{3}} (\tilde{u}_{n} \tilde{u}_{L,n}^{\frac{\p-2}{2}})^{\p} dx\right)^{\frac{2}{\p}}\leq C \beta^{2}\int_{\R^{3}} R^{\p-2} \tilde{u}^{\p}_{n} dx<\infty
\end{equation*}
and taking the limit as $L\rightarrow \infty$, we obtain $\tilde{u}_{n}\in L^{\frac{(\p)^{2}}{2}}(\R^{3})$.

\noindent
Now, noticing that $0\leq \tilde{u}_{L,n}\leq \tilde{u}_{n}$ and letting $L\rightarrow \infty$ in \eqref{simo1}, we have
\begin{equation*}
|\tilde{u}_{n}|_{\beta\p}^{2\beta}\leq C \beta^{2} \int_{\R^{3}} \tilde{u}^{\p+2(\beta-1)}_{n} dx,
\end{equation*}
from which we deduce that
\begin{equation*}
\left(\int_{\R^{3}} \tilde{u}^{\beta\p}_{n} dx\right)^{\frac{1}{\p(\beta-1)}}\leq (C \beta)^{\frac{1}{\beta-1}} \left(\int_{\R^{3}} \tilde{u}^{\p+2(\beta-1)}_{n} dx\right)^{\frac{1}{2(\beta-1)}}.
\end{equation*}
For $m\geq 1$ we define $\beta_{m+1}$ inductively so that $\p+2(\beta_{m+1}-1)=\p \beta_{m}$ and $\beta_{1}=\frac{\p}{2}$. Then we have
\begin{equation*}
\left(\int_{\R^{3}} \tilde{u}_{n}^{\beta_{m+1}\p} dx\right)^{\frac{1}{\p(\beta_{m+1}-1)}}\leq (C \beta_{m+1})^{\frac{1}{\beta_{m+1}-1}} \left(\int_{\R^{3}} \tilde{u}_{n}^{\p\beta_{m}}\,dx \right)^{\frac{1}{\p(\beta_{m}-1)}}.
\end{equation*}
Let us define
$$
D_{m}:=\left(\int_{\R^{3}} \tilde{u}_{n}^{\p\beta_{m}}\,dx\right)^{\frac{1}{\p(\beta_{m}-1)}}.
$$
A standard iteration argument shows that we can find $C_{0}>0$ independent of $m$ such that 
$$
D_{m+1}\leq \prod_{k=1}^{m} (C \beta_{k+1})^{\frac{1}{\beta_{k+1}-1}}  D_{1}\leq C_{0} D_{1}.
$$
Passing to the limit as $m\rightarrow \infty$ we get $|\tilde{u}_{n}|_{\infty}\leq K$ for all $n\in \mathbb{N}$.
\end{proof}

\noindent
Now, we give the proof of Theorem \ref{thm1}.
\begin{proof}[Proof of Theorem \ref{thm1}]
Firstly, we prove that there exists $\tilde{\e}_{0}>0$ such that for any $\e \in (0, \tilde{\e}_{0})$ and any mountain-pass solution $u_{\e} \in \h$ of \eqref{MP}, it results 
\begin{equation}\label{infty}
|u_{\e}|_{L^{\infty}(\R^{3}\setminus \Lambda_{\e})}<a_{0}. 
\end{equation}
Assume by contradiction that for some subsequence $\{\e_{n}\}_{n\in \mathbb{N}}$ such that $\e_{n}\rightarrow 0^{+}$, we can find $u_{n}:=u_{\e_{n}}\in \mathcal{H}_{\e_{n}}$ such that $\mathcal{J}_{\e_{n}}(u_{n})=c_{\e_{n}}$, $\mathcal{J}'_{\e_{n}}(u_{n})=0$ and 
\begin{equation}\label{eee}
|u_{n}|_{L^{\infty}(\R^{3}\setminus \Lambda_{\e_{n}})}\geq a_{0}.
\end{equation} 
From Lemma \ref{lem3.1N}, there exists $\{\tilde{y}_{n}\}_{n\in \mathbb{N}}\subset \R^{3}$ such that $\tilde{u}_{n}=u_{n}(\cdot+\tilde{y}_{n})\rightarrow \tilde{u}$ in $H^{s}(\R^{3})$ and $\e_{n}\tilde{y}_{n}\rightarrow y_{0}$ for some $y_{0}\in \Lambda$ such that $V(y_{0})=V_{0}$. 
Now, if we choose $r>0$ such that $\B_{r}(y_{0})\subset \B_{2r}(y_{0})\subset \Lambda$, we can see that $\B_{\frac{r}{\e_{n}}}(\frac{y_{0}}{\e_{n}})\subset \Lambda_{\e_{n}}$. 
Then, for any $y\in \B_{\frac{r}{\e_{n}}}(\tilde{y}_{n})$ it holds
\begin{equation*}
\left|y - \frac{y_{0}}{\e_{n}}\right| \leq |y- \tilde{y}_{n}|+ \left|\tilde{y}_{n} - \frac{y_{0}}{\e_{n}}\right|<\frac{1}{\e_{n}}(r+o_{n}(1))<\frac{2r}{\e_{n}}\, \mbox{ for } n \mbox{ sufficiently large. }
\end{equation*}
Hence, for these values of $n$ we have
\begin{equation}\label{ern}
\R^{3}\setminus \Lambda_{\e_{n}}\subset \R^{3} \setminus \B_{\frac{r}{\e_{n}}}(\tilde{y}_{n}).
\end{equation}
Now, we observe that $\tilde{u}_{n}$ is a solution to
\begin{align*}
(-\Delta)^{s} \tilde{u}_{n}+\tilde{u}_{n}=\xi_{n} \mbox{ in } \R^{3},
\end{align*}
where 
$$
\xi_{n}(x):=(a+b[\tilde{u}_{n}]^{2}_{s})^{-1}(g_{n}-V_{n} \tilde{u}_{n})+\tilde{u}_{n}
$$
and
$$
V_{n}(x):=V(\e_{n} x+\e_{n} \tilde{y}_{n}) \mbox{ and }  g_{n}(x):=g(\e_{n} x+\e_{n} \tilde{y}_{n}, \tilde{u}_{n}).
$$
Put 
$$
\xi(x):=(a+b[\tilde{u}]^{2}_{s})^{-1}[f(\tilde{u})+ |\tilde{u}|^{2^{*}_{s}-2}\tilde{u}-V(y_{0})\tilde{u}]+\tilde{u}.
$$ 
Using Lemma \ref{Moser}, the interpolation in the $L^{p}$ spaces, $\tilde{u}_{n}\rightarrow \tilde{u}$ in $H^{s}(\R^{3})$, assumptions $(g_1)$ and $(g_3)$ we can see that 
$$
\xi_{n}\rightarrow \xi \mbox{ in } L^{p}(\R^{3}) \quad \forall p\in [2, \infty), 
$$
and that there exists $C>0$ such that 
$$
|\xi_{n}|_{\infty}\leq C \quad \forall n\in \mathbb{N}.
$$
Consequently, $\tilde{u}_{n}(x)=(\mathcal{K}*\xi_{n})(x)=\int_{\R^{3}} \mathcal{K}(x-z) \xi_{n}(z) \,dz$, where 
$\mathcal{K}$  is the Bessel kernel and satisfies the following properties \cite{FQT}:
\begin{compactenum}[$(i)$]
\item $\mathcal{K}$ is positive, radially symmetric and smooth in $\R^{3}\setminus \{0\}$,
\item there is $C>0$ such that $\displaystyle{\mathcal{K}(x)\leq \frac{C}{|x|^{3+2s}}}$ for any $x\in \R^{3}\setminus \{0\}$,
\item $\mathcal{K}\in L^{r}(\R^{3})$ for any $r\in [1, \frac{3}{3-2s})$.
\end{compactenum} 
Hence, arguing as in Lemma $2.6$ in \cite{AM}, we can see that 
\begin{equation}\label{AM3}
\tilde{u}_{n}(x)\rightarrow 0 \mbox{ as } |x|\rightarrow \infty \mbox{ uniformly in } n\in \mathbb{N}.
\end{equation}
Therefore, we can find $R>0$ such that 
$$
\tilde{u}_{n}(x)<a_{0} \quad  \forall |x|\geq R \quad  \forall n\in \mathbb{N},
$$ 
which yields $u_{n}(x)<a_{0}$ for any $x\in \R^{3}\setminus \B_{R}(\tilde{y}_{n})$ and $n\in \mathbb{N}$. \\
On the other hand, there exists $\nu \in \mathbb{N}$ such that for any $n\geq \nu$, it holds
$$
\R^{3}\setminus \Lambda_{\e_{n}}\subset \R^{3} \setminus \B_{\frac{r}{\e_{n}}}(\tilde{y}_{n})\subset \R^{3}\setminus \B_{R}(\tilde{y}_{n}),
$$
which gives 
$$
u_{n}(x)<a_{0}  \quad \forall x\in \R^{3}\setminus \Lambda_{\e_{n}}.
$$
This last fact contradicts  \eqref{eee} and thus \eqref{infty} is verified.

Now, let $u_{\e}$ be a solution to \eqref{MP}. Since $u_{\e}$ satisfies \eqref{infty} for any $\e \in (0, \tilde{\e}_{0})$, it follows from the definition of $g$ that $u_{\e}$ is a solution to \eqref{Pe}, and then $\hat{u}_{\e}(x)=u(x/\e)$ is a solution to \eqref{P} for any $\e \in (0, \tilde{\e}_{0})$.

Finally, we study the behavior of the maximum points of solutions to problem \eqref{Pe}. Take $\e_{n}\rightarrow 0^{+}$ and consider a sequence $\{u_{n}\}_{n\in \mathbb{N}}\subset \mathcal{H}_{\e_{n}}$ of solutions to \eqref{Pe}.
We first notice that, by $(g_1)$, there exists $\gamma\in (0, a_{0})$ such that 
\begin{equation}\label{4.4FS}
g(\e_{n} x, t)t=f(t)t+t^{\p}\leq \frac{V_{1}}{K} t^{2} \quad \mbox{ for any } x\in \R^{3},\, 0\leq t\leq \gamma.
\end{equation}
The same argument as before yields, for some $R>0$,
\begin{equation}\label{4.5FS}
|u_{n}|_{L^{\infty}(\R^{3}\setminus \B_{R}(\tilde{y}_{n}))}<\gamma.
\end{equation}
Moreover, up to extract a subsequence, we may assume that 
\begin{equation}\label{4.6FS}
|u_{n}|_{L^{\infty}(\B_{R}(\tilde{y}_{n}))}\geq \gamma.
\end{equation}
Indeed, if \eqref{4.6FS} does not hold, we can see that \eqref{4.5FS} implies that $|u_{n}|_{\infty}<\gamma$. Then, in view of $\langle \mathcal{J}'_{\e_{n}}(u_{n}), u_{n}\rangle=0$ and \eqref{4.4FS}, we can see that
\begin{equation*}
 \|u_{n}\|_{\e_{n}}^{2}\leq \|u_{n}\|^{2}_{\e_{n}}+b[u_{n}]^{4}_{s}=\int_{\R^{3}} g(\e_{n} x, u_{n}) u_{n} \,dx\leq \frac{V_{1}}{K} \int_{\R^{3}} u_{n}^{2} \, dx
\end{equation*}
which gives $\|u_{n}\|_{\e_{n}}=0$, that is a contradiction.
Hence, \eqref{4.6FS} holds true. In the light of \eqref{4.5FS} and \eqref{4.6FS}, we can deduce that the maximum point $p_{n}\in \R^{3}$ of $u_{n}$ belongs to $\B_{R}(\tilde{y}_{n})$. Thus, $p_{n}=\tilde{y}_{n}+q_{n}$ for some $q_{n}\in \B_{R}$. 
Recalling that the solution to \eqref{P} is of the form $\hat{u}_{n}(x):=u_{n}(x/\e_{n})$, we conclude that the maximum point $\eta_{\e_{n}}$ of $\hat{u}_{n}$ is given by $\eta_{\e_{n}}:=\e_{n} \tilde{y}_{n}+\e_{n} q_{n}$. Since $\{q_{n}\}_{n\in \mathbb{N}}\subset \B_{R}$ is bounded and $\e_{n} \tilde{y}_{n}\rightarrow y_{0}$ with $V(y_{0})=V_{0}$, from the continuity of $V$ we can infer that
$$
\lim_{n\rightarrow \infty} V(\eta_{\e_{n}})=V(y_{0})=V_{0}.
$$
Next, we give a decay estimate for $\hat{u}_{n}$. Invoking Lemma $4.3$ in \cite{FQT}, we know that there exists a positive function $w$ such that 
\begin{align}\label{HZ1}
0<w(x)\leq \frac{C}{1+|x|^{3+2s}},
\end{align}
and
\begin{align}\label{HZ2}
(-\Delta)^{s} w+\frac{V_{1}}{2(a+bA_{1}^{2})}w\geq 0 \quad \mbox{ in } \R^{3}\setminus \B_{R_{1}}, 
\end{align}
for some suitable $R_{1}>0$, and $A_{1}>0$ is such that 
$$
a+b[u_{n}]^{2}_{s}\leq a+bA_{1}^{2} \quad \forall n\in \mathbb{N}.
$$ 
Using $(f_1)$, the definition of $g$ and \eqref{AM3}, we can find $R_{2}>0$ sufficiently large such that
\begin{align}\label{HZ3}
(-\Delta)^{s} \tilde{u}_{n}+\frac{V_{1}}{2(a+bA_{1}^{2})} \tilde{u}_{n} 
&\leq (-\Delta)^{s} \tilde{u}_{n}+\frac{V_{1}}{2(a+b[\tilde{u}_{n} ]^{2})} \tilde{u}_{n} \nonumber \\
&=\frac{1}{a+b[\tilde{u}_{n} ]^{2}_{s}}\left[g(\e_{n} x+\e_{n}y_{n}, \tilde{u}_{n})-\left(V_{n}-\frac{V_{1}}{2}\right)\tilde{u}_{n}\right]\nonumber \\
&\leq \frac{1}{a+b[\tilde{u}_{n}]^{2}_{s}}\left[g(\e_{n} x+\e_{n}y_{n}, \tilde{u}_{n})-\frac{V_{1}}{2}\tilde{u}_{n}\right]\leq 0 \mbox{ in } \R^{3}\setminus \B_{R_{2}}. 
\end{align}
Define $R_{3}:=\max\{R_{1}, R_{2}\}>0$ and we set 
\begin{align}\label{HZ4}
c:=\inf_{\B_{R_{3}}} w>0 \mbox{ and } \tilde{w}_{n}:=(d+1)w-c\tilde{u}_{n},
\end{align}
where $d:=\sup_{n\in \mathbb{N}} |\tilde{u}_{n}|_{\infty}<\infty$. 
In what follows, we show that 
\begin{equation}\label{HZ5}
\tilde{w}_{n}\geq 0 \mbox{ in } \R^{3}.
\end{equation}
Firstly, we can observe that \eqref{HZ2}, \eqref{HZ3} and \eqref{HZ4} yield
\begin{align}
&\tilde{w}_{n}\geq cd+w-cd>0 \mbox{ in } \B_{R_{3}} \label{HZ0},\\
&(-\Delta)^{s} \tilde{w}_{n}+\frac{V_{1}}{2(a+bA_{1}^{2})}\tilde{w}_{n}\geq 0 \mbox{ in } \R^{3}\setminus \B_{R_{3}} \label{HZ00}.
\end{align}
Now, we argue by contradiction and we assume that there exists a sequence $\{\bar{x}_{n, k}\}_{k\in \mathbb{N}}\subset \R^{3}$ such that 
\begin{align}\label{HZ6}
\inf_{x\in \R^{3}} \tilde{w}_{n}(x)=\lim_{k\rightarrow \infty} \tilde{w}_{n}(\bar{x}_{n, k})<0. 
\end{align}
By (\ref{AM3}), \eqref{HZ1} and the definition of $\tilde{w}_{n}$, it is clear that $|\tilde{w}_{n}(x)|\rightarrow 0$ as $|x|\rightarrow \infty$, uniformly in $n\in \mathbb{N}$. Thus, $\{\bar{x}_{n, k}\}_{k\in \mathbb{N}}$ is bounded, and, up to subsequence, we may assume that there exists $\bar{x}_{n}\in \R^{3}$ such that $\bar{x}_{n, k}\rightarrow \bar{x}_{n}$ as $k\rightarrow \infty$. 
It follows from (\ref{HZ6}) that
\begin{align}\label{HZ7}
\inf_{x\in \R^{3}} \tilde{w}_{n}(x)= \tilde{w}_{n}(\bar{x}_{n})<0.
\end{align}
From the minimality property of $\bar{x}_{n}$ and the representation formula for the fractional Laplacian \cite{DPV}, we can see that 
\begin{align}\label{HZ8}
(-\Delta)^{s}\tilde{w}_{n}(\bar{x}_{n})=\frac{C_{s}}{2} \int_{\R^{3}} \frac{2\tilde{w}_{n}(\bar{x}_{n})-\tilde{w}_{n}(\bar{x}_{n}+\xi)-\tilde{w}_{n}(\bar{x}_{n}-\xi)}{|\xi|^{3+2s}} d\xi\leq 0.
\end{align}
Taking into account (\ref{HZ0}) and (\ref{HZ6}) we can infer that $\bar{x}_{n}\in \R^{3}\setminus \B_{R_{3}}$.
This together with (\ref{HZ7}) and (\ref{HZ8}) implies
$$
(-\Delta)^{s} \tilde{w}_{n}(\bar{x}_{n})+\frac{V_{1}}{2(a+bA_{1}^{2})}\tilde{w}_{n}(\bar{x}_{n})<0,
$$
which is impossible in view of (\ref{HZ00}).
Hence, (\ref{HZ5})  is verified.

According to (\ref{HZ1}) and \eqref{HZ5}, we obtain
\begin{align}\label{HZ9}
0<\tilde{u}_{n}(x)\leq \frac{\tilde{C}}{1+|x|^{3+2s}} \quad  \forall n\in \mathbb{N} \quad \forall x\in \R^{3}, 
\end{align}
for some constant $\tilde{C}>0$.
Since $\hat{u}_{n}(x)=u_{n}(\frac{x}{\e_{n}})=\tilde{u}_{n}(\frac{x}{\e_{n}}-\tilde{y}_{n})$ and $\eta_{\e_{n}}=\e_{n}\tilde{y}_{n}+\e_{n} q_{n}$, we can use (\ref{HZ9}) to deduce that 
\begin{align*}
0<\hat{u}_{n}(x)&=u_{n}\left(\frac{x}{\e_{n}}\right)=\tilde{u}_{n}\left(\frac{x}{\e_{n}}-\tilde{y}_{n}\right) \\
&\leq \frac{\tilde{C}}{1+|\frac{x}{\e_{n}}-\tilde{y}_{n}|^{3+2s}} \\
&=\frac{\tilde{C} \e_{n}^{3+2s}}{\e_{n}^{3+2s}+|x- \e_{n} \tilde{y}_{n}|^{3+2s}} \\
&\leq \frac{\tilde{C} \e_{n}^{3+2s}}{\e_{n}^{3+2s}+|x-\eta_{\e_{n}}|^{3+2s}} \quad \forall x\in \R^{3}.
\end{align*}
This ends the proof of Theorem \ref{thm1}.
\end{proof}

\end{document}